\documentclass[12pt]{amsart}
\usepackage{pifont}
\usepackage{txfonts}
\usepackage{}
\usepackage{amsfonts}
\usepackage{graphicx}
\usepackage{epstopdf}
\usepackage{amsmath}
\usepackage{amsthm}
\usepackage{latexsym,bm}
\usepackage{amssymb}
\usepackage{esint}
\usepackage{enumitem}
\usepackage{indentfirst}
\usepackage{amscd}
\newtheorem{thm}{Theorem}[section]
\newtheorem{cor}[thm]{Corollary}
\newtheorem{lem}[thm]{Lemma}
\newtheorem{prop}[thm]{Proposition}
\newtheorem{defn}[thm]{Definition}

\newtheorem{Remark}{Remark}
\numberwithin{equation}{section}
\numberwithin{Remark}{section}

\begin{document}

\title{On Uniqueness of Conformally Compact Einstein Metrics with Homogeneous Conformal Infinity}

\author{Gang Li$^\dag$}

\begin{abstract} In this paper we show that for a Berger metric $\hat{g}$ on $S^3$, the non-positively curved conformally compact Einstein metric on the $4$-ball $B_1(0)$ with $(S^3, [\hat{g}])$ as its conformal infinity is unique up to isometries and it is the metric constructed by Pedersen \cite{Pedersen}. In particular, since in \cite{LiQingShi}, we proved that if the Yamabe constant of the conformal infinity $Y(S^3, [\hat{g}])$ is close to that of the round sphere then any conformally compact Einstein manifold filled in must be negatively curved and simply connected, therefore if $\hat{g}$ is a Berger metric on $S^3$ with $Y(S^3, [\hat{g}])$ close to that of the round metric, the conformally compact Einstein metric filled in is unique up to isometries.
\end{abstract}

\renewcommand{\subjclassname}{\textup{2000} Mathematics Subject Classification}


\thanks{$^\dag$ Research supported by The Fundamental Research Funds of Shandong University 2016HW008.}

\address{Gang Li, Department of Mathematics, Shandong University, Jinan, Shandong Province, China}
\email{runxing3@gmail.com}

\maketitle


\section{Introduction}

The research on conformally compact Einstein manifolds has been an active area since the paper \cite{FG}. In this paper we are concerned with uniqueness of conformally compact Einstein metrics( for definition, see Definition \ref{defn_conformallycompactEinsteinmetric}) on manifolds with prescribed conformal infinity.

In \cite{GL}, given any Riemannian metric on the $n$-sphere $S^n$ which is $C^{2,\alpha}$ close to the round metric as the conformal infinity, Graham and Lee proved the existence of a conformally compact Einstein metric on the $(n+1)$-ball $B_1(0)$, which is unique in a small neighborhood of the asymptotic solution they constructed in a weighted space by the implicit function theorem. It is interesting to understand whether the solution is globally unique with the prescribed conformal infinity. On the other hand, in light of LeBrun's local construction in \cite{LeBrun1}, when the conformal infinity is a Berger metric on $S^3$ or a generalized Berger metric which is left invariant under the $\text{SU}(2)$ action, Pedersen \cite{Pedersen} and Hitchin \cite{Hitchin1} constructed a global conformally compact Einstein metric on the $4$-ball, which has self-dual Weyl curvature, and the metric is unique under the self-duality assumption. When the conformal structure at infinity and the non-local term in the expansion of the Einstein metric at infinity are both given, Anderson \cite{Anderson} and Biquard \cite{Biquard2} proved that the conformally compact Einstein metric is unique up to isometry. Recall that if the conformal infinity is the conformal class of the round sphere metric, it is proved that the conformally compact Einstein metric must be the hyperbolic space, see \cite{Andersson-Dahl}\cite{Q}\cite{DJ}\cite{LiQingShi}. Based on \cite{ST} and \cite{DJ}, in \cite{LiQingShi} we proved that for any conformal infinity $(S^n, [\hat{g}])$ with its Yamabe constant close to that of the round sphere metric, the conformally compact Einstein manifold filled in must be a Hadamard manifold, with its sectional curvature close to $-1$ uniformly.

In X. Wang's paper \cite{Wang}, he discussed existence of Killing vector fields on a non-positively curved conformally compact Einstein manifold by constructing an asymptotically Killing vector field, see in Theorem \ref{thm_baseVector} and Proposition \ref{prop_vector_extension}. That gives possibility of proving existence of Killing vector fields under conditions near infinity, and possibility of the proof of global uniqueness of conformally compact Einstein metrics with prescribed conformal infinity with symmetry. Along this line we prove the following main theorem of the paper:
\begin{thm}\label{thm_someBergermetric}
Let $\hat{g}$ be a Berger metric on $S^3$ so that $\hat{g}$ has the diagonal form
\begin{align*}
\hat{g}=\lambda_1\sigma_1^2+\lambda_2(\sigma_2^2+\sigma_3^2),
\end{align*}
where $\lambda_1$ and $\lambda_2$ are two positive constants and $\sigma_1,\sigma_2$ and $\sigma_3$ are three left invariant $1$-forms under the $\text{SU}(2)$ action. Assume that $\frac{1}{4}<\frac{\lambda_1}{\lambda_2}<4$, then up to isometry there exists at most one non-positively curved conformally compact Einstein metric on the $4$-ball $B_1(0)$ with $(S^3, [\hat{g}])$ as its conformal infinity. In particular, it is the metric constructed in \cite{Pedersen} when it is non-positively curved. Also, it is the perturbation metric in \cite{GL} when $\frac{\lambda_1}{\lambda_2}$ is close to $1$. Moreover, by the theorem in \cite{LiQingShi}, for $\frac{\lambda_1}{\lambda_2}$ close enough to $1$, any conformally compact Einstein manifold filled in is automatically negatively curved and simply connected, and therefore it is unique up to isometry.
\end{thm}
By the main theorem, when the  sectional curvature of Pedersen's metric in \cite{Pedersen} is not non-positive at some point, there exists no non-positively curved conformally compact Einstein metric on the $4$-ball $B_1(0)$ with such conformal infinity.

Let $(M^{n+1}, g)$ be a simply connected non-positively curved conformally compact Einstein manifold. Under Wang's suggestion in \cite{Wang}, we show that for any smooth conformal Killing vector field $Y$ at the conformal infinity $(\partial M, [\hat{g}])$, it extends to a unique Killing vector field on $(M,g)$, see Theorem \ref{thm_symmetryrealization} and Lemma \ref{lem_zeroset}. To do this, by a lengthy but smart calculation based on the expansion of the Einstein metric $(\ref{equn_expansion1})(\ref{equn_expansion2})$, we show that the vector field $(\ref{equn_vector1})$ with $(\ref{vect1_tant})$ and $(\ref{vect1_vert})$ relating to the conformal Killing vector field on $(\partial M, [\hat{g}])$ is indeed an asymptotically Killing vector field. Therefore by Theorem \ref{thm_baseVector}, $Y$ extends to a Killing vector field $X$ on $(M,g)$.

We start with a unique center of gravity $p_0$ of $(M, g)$, which is a common fixed point of all isometries on the manifold. We then show that for any one parameter isometry group generated by a Killing vector field $X$ on $(M, g)$, the set of fixed points is in fact the image of a sub-space of the tangent space under the exponential map at the center of gravity $p_0$, see Lemma \ref{lem_zeroset}. Then based on the extension of the Killing vector fields we prove that if the conformal infinity 
is $(S^n, [\hat{g}])$ with $\hat{g}$ a homogeneous metric, i.e. there exist Killing vector fields $Y_1,...,Y_{n+k}$ which give a basis of the tangent space at each point on the boundary $(S^n,\hat{g})$, then the geodesic defining function $x$ about $\hat{g}$ and the distance function $r$ to the center of gravity satisfy
\begin{align*}
x=Ce^{-r}
\end{align*}
for some constant $C>0$, where $0\leq x\leq C$, see Theorem \ref{thm_definingfunction}. Homogeneous conformal infinity data guarantees that the geodesic spheres centered at the center of gravity $p_0\in M$ are also homogeneous spaces and are invariant under the action of the group generated by the Killing vector fields. For a fixed $q\in \partial M$, along the geodesic connecting $q$ and $p_0$ we show that the boundary value problem of the Einstein equations with non-positive sectional curvature and a homogeneous conformal infinity is equivalent to a boundary value problem of a system of ordinary differential equations $(\ref{equn_EinsteinODEs1})-(\ref{equn_boundaryvalue1})$. This gives a natural way of fixing gauge for the Einstein equations. Note that there is a nice description about classification of homogeneous metrics on the sphere $S^n$ in \cite{Ziller1}. In particular, if the conformal infinity has a representation which is a Berger metric on $S^3$, we show in Lemma \ref{lem_symmetry_Threedimensional} that the metric restricted on each geodesic sphere centered at the center of gravity is diagonal, and the problem becomes $(\ref{equn_BergerEinstein01})-(\ref{equn_BergerBV01})$; while for a generalized Berger metric on $S^3$, the problem reduces to the boundary value problem $(\ref{equn_GBergerEinstein01})-(\ref{equn_GBergerBV01})$. Direct calculation shows that Lemma \ref{lem_symmetry_Threedimensional} does not lead to self-duality or anti-self-duality condition of the Einstein metric. Using a comparison argument, we prove that the problem $(\ref{equn_BergerEinstein01})-(\ref{equn_BergerBV01})$ has a unique solution for $\frac{1}{4}<\phi(0)<4$ with $\phi(0)\neq 1$, see Theorem \ref{thm_BergermetricStability}. Uniqueness of the conformally compact Einstein metric with prescribed conformal infinity $(S^n, [\hat{g}])$ where $\hat{g}$ is a generalized Berger metric on $S^3$ or a homogeneous metric on $S^n$( $n\geq 4$) will be discussed else where. 
\vskip0.2cm
{\bf Acknowledgements.} The author would like to thank Professor Jie Qing and Professor Yuguang Shi for introducing him to this area, their helpful discussion and constant support.  The author is grateful to Professor Xiaodong Wang for discussion on existence of Killing vector fields on conformally compact Einstein manifolds. Thanks also due to Professor John M. Lee for discussion on regularity of conformally compact Einstein metrics. The author is thankful to Professor Wolfgang Ziller for pointing him the reference \cite{Ziller1}. The author would like to thank Jianquan Ge and Xiaoyang Chen for helpful discussion.

\section{A Fixed Point Discussion.}

Let $(M^{n+1}, g)$ be a Cartan-Hadamard manifold, i.e., $(M^{n+1}, g)$ is a simply-connected non-positively curved complete Riemannian manifold. It is well-known that $M^{n+1}$ is diffeomorphic to the Euclidean space $\mathbb{R}^{n+1}$. Let $p$ be a point in $M$. Now we assume also that $(M, g)$ is not locally conformally flat and for any point $q\in M$, the normal of the Weyl tensor $|W|_g(q)\to 0$ uniformly as the distance $d_g(q, p)\to \infty$. Denote
\begin{align*}
C=\{q\in M, |W|_g(q)=\sup_M|W|_g.\}.
\end{align*}
We know immediately that $C$ is a compact subset of $M$. Since $(M, g)$ is Cartan-Hadamard, by a theorem of Cartan( see \cite{Eberlein}), for every bounded subset $A \subset M$, there is a unique closed geodesic ball of the smallest radius that contains $A$. We assume $B(p_0)$ to be the unique closed geodesic ball of the smallest radius that contains $C$ and its center $p_0$ is called the {\it spherical center of gravity} of $C$, which is uniquely determined on $(M, g)$. So we also call $p_0$ the {\it spherical center of gravity( or center of gravity)} of $(M, g)$. Under any isometry on $(M, g)$, $C$ must be an invariant subset, and therefore, $p_0$ is a common fixed point for all isometries on $(M, g)$. It is clear that the geodesic spheres $S_{p_0}(r)$ centered at $p_0$ with radius $r>0$ are all invariant subsets under isometries on $(M, g)$.


On a Cartan-Hadamard manifold, for any two given points there exists a unique geodesic crossing them. All the geodesics are short geodesics. It is easy to check the following lemma:

\begin{lem}\label{lem_stableset}
Let $F: (M, g) \to (M, g)$ be an isometry. Then for any two fixed points $p_1, p_2$ under $F$, the geodesic crossing $p_1$ and $p_2$ must be a set of fixed points under $F$. Therefore, the set of fixed points under $F$ is either the single point $p_0$ which is the {\it spherical center of gravity} of $(M, g)$, or a smooth complete sub-manifold $Exp_{p_0}(S)$ of $M$, where $Exp_{p_0}$ is the exponential map at $p_0$ and $S$ is a linear subspace of the tangent space $T_{p_0}M$.
\end{lem}
For any closed subset $C_0$ of the zero set of the Killing vector $X$ in $M$ and any constant $r>0$, the set $S_{C_0}(r)=\{p\in M, dist_g(p, C_0)=r\}$ is invariant under the one-parameter group action induced by $X$.

Let $(r, \theta)=(r, \theta^1,...,\theta^n)$ be the polar coordinate on $T_{p_0}M$ with $(1, \theta^1,...,\theta^n)$ on the unit sphere of $T_{p_0}M$. Denote $\theta^0=r$. By the exponential map $Exp_{p_0}$, $(r, \theta)$ is considered as a polar coordinate on $M$. We can show that a Killing vector field $X$ is independent of $r$ under the coordinate $(r, \theta)$.
\begin{lem}\label{lem_vectorposition}
A Killing vector field $X$ has the form $X=\displaystyle\sum_{i=1}^n X^i(\theta)\frac{\partial}{\partial \theta^i}$ under the polar coordinate centered at $p_0$. That is, $X$ depends only on $\theta$ under the polar coordinate $(r, \theta^1,...,\theta^n)$.
\end{lem}
\begin{proof}
 A Killing vector field $X$ on $(M, g)$ must vanish at $p_0$ and it is orthogonal to $\frac{\partial}{\partial r}$ at any point $q\in M\setminus\{p_0\}$. Therefore, $X$ has the form
 \begin{align*}
 X=\sum_{i=1}^nX^i(r,\theta)\frac{\partial}{\partial \theta^i}.
 \end{align*}
 Under $(r, \theta)$, the metric $g$ has the expression
 \begin{align*}
 g=g_{00}dr^2+\sum_{i,j=1}^ng_{ij} d \theta^i d \theta^j,
 \end{align*}
 with $g_{00}=1$ and $g_{0i}=g_{i0}=0$ for $1\leq i \leq n$. The Christoffel symbols
 \begin{align*}
 &\Gamma_{00}^a(g)=\Gamma_{0a}^0(g)=\Gamma_{a0}^0(g)=0,\,\,\,0\leq a\leq n,\\
 &\Gamma_{0i}^j(g)=\frac{1}{2}\sum_{k=1}^ng^{jk}\frac{\partial}{\partial r}g_{ik},\,\,1\leq i \leq n.
 \end{align*}
 Let $X_a=\displaystyle\sum_{b=0}^ng_{ab}X^b$ for $0\leq a\leq n$ so that $X_0=0$. The condition
 \begin{align*}
 \nabla_aX_b+\nabla_bX_a=0,\,\,\,a,\,b=0,...,n,
 \end{align*}
 with $a=0$ and $b\geq 1$ gives
 \begin{align*}
 \frac{\partial}{\partial r}X_b+\frac{\partial}{\partial \theta^b}X_0 - 2\sum_{c=0}^n\Gamma_{0b}^cX_c=0.
 \end{align*}
 That is
 \begin{align*}
 \sum_{k=1}^ng_{bk}\frac{\partial}{\partial r}X^k=0,
 \end{align*}
 for $1\leq b\leq n$. We now have
 \begin{align*}
 \frac{\partial}{\partial r}X^k=0
 \end{align*}
 so that $X^k$ is independent of $r$ for $1\leq k \leq n$.
\end{proof}

\begin{defn}\label{defn_conformallycompactEinsteinmetric}
Suppose $\overline{M}$ is a smooth compact manifold with boundary, with $M$ its interior and $\partial M$ its boundary. A smooth defining function $x$ on $\overline{M}$ is a smooth function $x$ on $\overline{M}$ such that $x>0$ in $M$, $x=0$ and $dx\neq 0$ on $\partial M$. A complete Riemannian metric $g$ on $M$ is said to be conformally compact if there exists a smooth defining function $x$ such that $x^2g$ extends by continuity to a Riemannian metric( of class at least $C^0$) on $\overline{M}$. The rescaled metric $\bar{g}=x^2g$ is called a conformal compactification of $g$. If for some smooth defining function $x$, $\bar{g}$ is in $C^k(\overline{M})$ or the Holder space $C^{k,\alpha}(\overline{M})$, we say $g$ is conformally compact of class $C^k$ or $C^{k, \alpha}$. Moreover, if $g$ is also Einstein, we call $g$ a conformally compact Einstein metric. Also, for the restricted metric $\hat{g}=\bar{g}\big|_{\partial M}$, the conformal class $(\partial M, [\hat{g}])$ is called the conformal infinity of $(M, g)$. A defining function $x$ is called a geodesic defining function about $\hat{g}$ if $\hat{g}=\bar{g}\big|_{\partial M}$ and $|dx|_{\bar{g}}=1$ in a neighborhood of the boundary.
\end{defn}
Recall that for any smooth metric $h\in[\hat{g}]$ at the conformal infinity, there exists a unique geodesic defining function $x$ about $h$ in a neighborhood of $\partial M$, see \cite{Graham}. For a conformally compact Einstein metric of $C^2$, In \cite{CDLS}, based on \cite{Graham} the authors proved the following regularity result.

\begin{thm}
Assume $\overline{M}$ is a smooth compact manifold of dimension $n+1$, $n\geq 3$, with $M$ its interior and $\partial M$ its boundary. If $g$ is a conformally compact Einstein metric of class $C^2$ on $M$ with conformal infinity $(\partial M, [\gamma])$, and $\hat{g}\in [\gamma]$ is a smooth metric on $\partial M$. Then there exists  a smooth coordinates cover of $\overline{M}$ and a smooth geodesic defining $x$ corresponding to $\hat{g}$. Under this smooth coordinates cover, the conformal compactification $\bar{g}=x^2g$ is smooth up to the boundary for $n$ odd and has the expansion
\begin{align}\label{equn_expansion1}
\bar{g}=\,dx^2+g_x=\, dx^2+\hat{g}+x^2g^{(2)}+\,(\text{even powers})\,+x^{n-1} g^{(n-1)}+ x^{n}g^{(n)}+...
\end{align}
with $g^{(k)}$ smooth symmetric $(0,2)$-tensors on $\partial M$ such that for $2k<n$, $g^{(2k)}$ can be calculated explicitly inductively using the Einstein equations and $g^{(n)}$ is a smooth trace-free nonlocal term; while for $n$ even,  $\bar{g}$ is of class $C^{n-1}$, and more precisely it is polyhomogeneous and has the expansion
\begin{align}\label{equn_expansion2}
\bar{g}=\,dx^2+g_x=\, dx^2+\hat{g}+x^2g^{(2)}+\,(\text{even powers})\,+x^n\log(x) \tilde{g}+ x^{n}g^{(n)}+...
\end{align}
with $\tilde{g}$ and $g^{(k)}$ smooth symmetric $(0,2)$-tensors on $\partial M$, such that for $2k<n$, $g^{(2k)}$ and $\tilde{g}$ can be calculated explicitly inductively using the Einstein equations, $\tilde{g}$ is trace-free and $g^{(n)}$ is a smooth nonlocal term with its trace locally determined.
\end{thm}
For instance, with $R_{\hat{g}}$ the scalar curvature and $R_{ij}(\hat{g})$ the Ricci curvature tensor of $\hat{g}$ we have
\begin{align}
g_{ij}^{(2)}=\frac{R_{\hat{g}}}{2(n-1)(n-2)}\hat{g}_{ij}-\frac{R_{ij}(\hat{g})}{n-2}.
\end{align}

Let $(M^{n+1}, g)$ be a Cartan-Hadamard manifold which is conformally compact Einstein with conformal infinity $(\partial M, [\hat{g}])$.  Let $x$ be the
smooth geodesic defining function. Near infinity we use the local coordinates $(x^0, x^1,..., x^n)=(x, x^1, ..., x^n)=(x,y)$, with $(x^1, ..., x^n)$ local coordinates on $\partial M$. Under the local coordinates, the metric can be expressed as
\begin{align}\label{equn_compactificationmetric}
g=x^{-2}(dx^2+g_x)=x^{-2}(dx^2+\sum_{i,j\geq 1}h_{ij}(x,y)dx^idx^j).
\end{align}
We have the useful proposition:
\begin{prop}\label{prop_vector_extension}
(Proposition 4.1., \cite{Wang}) Let $X$ be a Killing vector field on $(M, g)$. Then $X$ extends to a smooth vector field on $\overline{M}$ whose restriction on $\partial M$ is a conformal Killing vector field of $(\partial M, [\hat{g}])$.
\end{prop}
Based on Lemma \ref{lem_stableset} and Proposition \ref{prop_vector_extension}, one easily obtains that
\begin{lem}\label{lem_zeroset}
Let $X$ be a Killing vector field on $(M, g)$. Then for any two zero points $p_1, p_2$ of $X$, the geodesic crossing $p_1$ and $p_2$ must be a set of zero points of $X$. Therefore, the set of zero points of $X$ is either the single point $p_0$ which is the {\it spherical center of gravity} of $(M, g)$, or the closure of a smooth complete sub-manifold $Exp_{p_0}(S)$ of $M$ on $\overline{M}$, where $Exp_{p_0}$ is the exponential map at $p_0$ and $S$ is a linear subspace of the tangent space $T_{p_0}M$.
\end{lem}

We now choose $M$ to be the interior of the manifold with boundary $\overline{M}$ which is diffeomorphic to the closed unit ball in the Euclidean space $\mathbb{R}^{n+1}$, and $(M^{n+1}, g)$ is a Cartan-Hadamard manifold which is conformally compact Einstein, with $(\partial M, [\hat{g}])$ the conformal infinity. Let $x$ be the smooth geodesic defining function. Let $Y$ be a conformal Killing vector field on $(\partial M, \hat{g})$. We will show that $Y$ can be extended to a Killing vector field in $(M, g)$.




\section{Killing Vector Fields on Non-positively Curved Conformally Compact Einstein Manifolds}

\begin{defn}
Let $(M,g)$ be a conformally compact Einstein manifold and $x$ is a defining function. A vector field $V$ on $\overline{M}$ is asymptotically Killing if the Lie derivative $L_V g= O(x^{n-2})$ 
as $x\to 0$. That is to say, $|L_Vg|_g=O(x^n)$.
\end{defn}

In \cite{Wang}, it is proved that

\begin{thm}\label{thm_baseVector}
(Theorem 4.1.,\cite{Wang}) Let $M^{n+1}$ be the interior of a closed manifold $\overline{M}$ with boundary $\partial M$. Assume $(M^{n+1}, g)$ is a Cartan-Hadamard manifold which is conformally compact Einstein. Let $x$ be a geodesic defining function. For any asymptotic Killing vector field $V$ on $\overline{M}$, there is a Killing vector field $X$ such that $X\big|_{\partial M}=V\big|_{\partial M}$. Moreover, if $V$ is smooth up to $\partial M$, $X$ is at least of $C^{n+1}$ up to $\partial M$ with the expansion
\begin{align*}
X=\sum_{k=0}^nX^k(x,y)\frac{\partial}{\partial x^k}=\sum_{k=0}^n(\sum_{m=0}^{n+1}(X^k)^{(m)}(y)x^m+o(x^{n+1}))\frac{\partial}{\partial x^k},
\end{align*}
with $(X^k)^{(m)}(y)$ a smooth function on the boundary which can be solved explicitly inductively for $0\leq m\leq n+1$ under the local coordinate $(x^0,x^1,...,x^n)=(x,y)$ near the boundary and $(X^0)^{(0)}=0$. Here $y=(x^1,..., x^n)$ is a local coordinate on the boundary. 
\end{thm}
By Theorem \ref{thm_baseVector}, to show that $Y$ can be extended to a Killing vector field in $(M, g)$, we only need to extend $Y$ to an asymptotically Killing vector field on $\overline{M}$. In fact, Theorem \ref{thm_baseVector} still holds if we assume the decay rate to be $L_V g= O(x^{s})$ 
for some $\frac{n}{2}-2 < s < n-1$ in the definition of the {\em asymptotically Killing vector field} $V$. That is, $|L_Vg|_g =O(x^{s+2})$. Indeed, for this issue there are two places involved in the proof of Theorem 4.1. in \cite{Wang}:

First, to construct the Killing vector field, we need to use elliptic edge operator theory to solve a vector field $Z$ satisfying
\begin{align*}
\Delta_gZ-n Z=-(\Delta V - n V).
\end{align*}
Let $\tilde{V}=g_{ij}V^i\, dx^j$ be the $1$-form corresponding to the vector $V=V^i\,\frac{\partial}{\partial x^i}$. Locally near the boundary, denote $\tilde{V}=V_i \,dx^i$. For any constant $-1 <s<n-1$, if $L_Vg=O(x^s)$, 
since $Ric_g=-ng$, we have
\begin{align*}
(\Delta_g - n)V_i =\nabla^p\nabla_pV_i+\nabla^p\nabla_iV_p- \nabla_i\text{div}_gV=(\delta L_Vg)_i- \nabla_i\text{div}_gV = O(x^{s+1}).
\end{align*}
Here we have use the fact $\text{div}_gV=\frac{1}{2}g^{ij}L_Vg_{ij}$. Then by elliptic edge operator theory on conformally compact manifolds with non-positive sectional curvature e.g., see \cite{Lee} \cite{Mazzeo1}, there exists a $1$-form $\tilde{Z}$ solving the equation
\begin{align*}
(\Delta_g-n)\tilde{Z}=(\Delta_g - n)\tilde{V},
\end{align*}
where $\tilde{Z}=Z_i\,dx^i$ with $Z_i=O(x^{s+1})$. Here $Z_i$ has smooth expansion about $x$ with possible $log(x)$ terms since the order $x^n\log(x)$( it is polyhomogeneous), and so it is of at least $C^{n-1, \alpha}$ up to the boundary for $0 < \alpha < 1$ with smooth coefficients of $x^k$ for $k\leq n-1$ and of the possible term $x^n\log(x)$ which can be solved explicitly. Let $Z=g^{ij}Z_i\,\frac{\partial}{\partial x^j}$ be the vector field corresponding to $\tilde{Z}$, then $Z$ is at least of $C^{n+1, \alpha}(\overline{M})$.

Second, to show the vector field $X=V+Z$ is the Killing vector field in need, the author used some integration formulae and inequalities, which only requires that $s>\frac{n}{2}-2$.

Therefore, for a given conformal Killing vector field $Y$ on $(\partial M, [\hat{g}])$, in order to find a Killing vector field $X$ on $(M, g)$ so that $X\big|_{\partial M}=Y$, we only need to construct a vector field $Z$ on $\overline{M}$ such that $Z\big|_{\partial M}=Y$ and the Lie derivative
\begin{align}\label{decay_indices}
L_Zg=O(x^s)
\end{align}
with some $s>\frac{n}{2}-2$.

Let $\hat{g}$ be a representation of the conformal infinity and $x$ be the geodesic defining function about $\hat{g}$. Let $Z$ be a smooth vector field on $\overline{M}$ such that $Z\big|_{\partial M} = Y$. Assume that $Z$ is expressed as
\begin{align}\label{equn_vector1}
Z=a(x,y)\frac{\partial}{\partial x}+ \sum_{i\geq 1}b^i(x,y)\frac{\partial}{\partial x^i}
\end{align}
in the coordinate $(x, y)=(x,x^1,...,x^n)$ near $\partial M$. Therefore,
\begin{align*}
&a(0,y)=0,\\
&Y =Y^k \frac{\partial}{\partial x^k}=b^k(0, y) \frac{\partial}{\partial x^k}.
\end{align*}
As in \cite{Wang}, a direct calculation gives the formulae
\begin{align}
&\label{equn_component1}L_Zg(\frac{\partial}{\partial x}, \frac{\partial}{\partial x})=2x^{-2}(\frac{\partial}{\partial x}a(x,y) -\frac{a(x,y)}{x}),\\
&\label{equn_component2}L_Zg(\frac{\partial}{\partial x}, \frac{\partial}{\partial x^i})=x^{-2}(\frac{\partial}{\partial x}b^j(x,y)h_{ij}+ \frac{\partial a(x,y)}{\partial x^i}),\\
&\label{equn_component3}L_Zg(\frac{\partial}{\partial x^i}, \frac{\partial}{\partial x^j})=(\frac{\partial b^k}{\partial x^i}+b^s\Gamma_{is}^k)\frac{h_{kj}}{x^2}+(\frac{\partial b^k}{\partial x^j}+ b^s \Gamma_{js}^k)\frac{h_{ki}}{x^2}-\frac{2 a}{x^3}h_{ij}+\frac{a(x,y)}{x^2}\frac{\partial h_{ij}}{\partial x},
\end{align}
where $\Gamma_{ij}^k$ is the Christoffel symbol of the metric $g_x=h_{ij} dx^idx^j$ in $(\ref{equn_compactificationmetric})$. We use the form of $Z$ in Proposition 4.1 in \cite{Wang} i.e., let
\begin{align}
&\label{vect1_vert}a(x,y)=x\,a_0(y),\\
&\label{vect1_tant}b^k(x,y)=b^k(0, y)-\int_0^x\sum_{i\geq 1}\,t\,\frac{\partial a_0(y)}{\partial x^i}\,h^{ik}(t, y) dt,
\end{align}
where $a_0(y)=\frac{1}{n}\text{div}_{\hat{g}}Y$. It is easy to check that $(\ref{equn_component1})$ and $(\ref{equn_component2})$ vanish identically. So we only need to handle $(\ref{equn_component3})$. We will use the expansion $(\ref{equn_expansion1})$ and $(\ref{equn_expansion2})$ of the metric $g_x$ near the boundary and that
\begin{align*}
L_Y\hat{g}=\frac{\text{div}_{\hat{g}}Y}{n}\hat{g}
\end{align*}
to show that the right hand side of $(\ref{equn_component3})$ is of order $O(x^{n-2})$.


By substituting $(\ref{vect1_vert})$ and $(\ref{vect1_tant})$ to $(\ref{equn_component3})$ we have
\begin{align*}
&L_Zg(\frac{\partial}{\partial x^i}, \frac{\partial}{\partial x^j})\\
=\,&\big[\,\frac{\partial b^k(0,y)}{\partial x^i}-\int_0^x \frac{\partial}{\partial x^i}(\frac{\partial a_0(y)}{\partial x^s}\,h^{sk}(t,y))\,t\,dt+\big(b^s(0,y)- \int_0^x \frac{\partial a_0(y)}{\partial x^m}\,h^{ms}(t,y)\,t\,dt\big)\,\Gamma_{is}^k(g_x)\,\big]\frac{h_{kj}}{x^2}\\
+\,&\big[\,\frac{\partial b^k(0,y)}{\partial x^j}-\int_0^x \frac{\partial}{\partial x^j}(\frac{\partial a_0(y)}{\partial x^s}\,h^{sk}(t,y))\,t\,dt+\big(b^s(0,y)- \int_0^x\frac{\partial a_0(y)}{\partial x^m}\,h^{ms}(t,y)\,t\,dt\big)\,\Gamma_{js}^k(g_x)\,\big]\frac{h_{ki}}{x^2}\\
-\,&\frac{2 a_0(y)}{x^2}h_{ij}+\frac{a_0(y)}{x}\frac{\partial h_{ij}}{\partial x}.
\end{align*}
Then substituting the expansion $(\ref{equn_expansion1})$ and $(\ref{equn_expansion2})$ of the metric $g_x$ we have
\begin{align*}
&L_Zg(\frac{\partial}{\partial x^i}, \frac{\partial}{\partial x^j})\\
&=x^{-2}(\nabla_i^{\hat{g}}b^k(0,y)\hat{g}_{kj}+\nabla_j^{\hat{g}}b^k(0,y)\hat{g}_{ki}- 2a_0(y)\hat{g}_{ij})\\
&+\nabla_i^{\hat{g}}b^k(0,y)g^{(2)}_{kj}+ \nabla_j^{\hat{g}}b^k(0,y)g^{(2)}_{ki}-\nabla_i^{\hat{g}}\nabla_j^{\hat{g}}a_0(y)+b^s(0,y)\nabla_s^{\hat{g}}g_{ij}^{(2)}+O(x)\\ 
&=x^{-2}L_Y(\hat{g}(\frac{\partial}{\partial x^i}, \frac{\partial}{\partial x^j}))+\nabla_i^{\hat{g}}b^k(0,y)g^{(2)}_{kj}+ \nabla_j^{\hat{g}}b^k(0,y)g^{(2)}_{ki}-\nabla_i^{\hat{g}}\nabla_j^{\hat{g}}a_0(y)+b^s(0,y)\nabla_s^{\hat{g}}g_{ij}^{(2)}+O(x)\\
&=\nabla_i^{\hat{g}}b^k(0,y)g^{(2)}_{kj}+ \nabla_j^{\hat{g}}b^k(0,y)g^{(2)}_{ki}-\nabla_i^{\hat{g}}\nabla_j^{\hat{g}}a_0(y)+b^s(0,y)\nabla_s^{\hat{g}}g_{ij}^{(2)}+O(x)\\
&=\frac{1}{n-2}[\frac{R_{\hat{g}}}{2(n-1)}L_Y\hat{g}_{ij}+\frac{L_YR_{\hat{g}}}{2(n-1)}g_{ij}-L_YR_{ij}(\hat{g})]+O(x).
\end{align*}
Here we have used that $Y$ is a conformal Killing vector field. To simplify the calculation, we use the following theorem, see \cite{Obata1} etc.

\begin{thm}\label{thm_isometry_rigidity}
(Obata, \cite{Obata1}) Let $(N, g)$ be a closed smooth Riemannian manifold of dimension $n$, which is not conformally equivalent to the round sphere. There exists a smooth metric $h\in [g]$ so that all smooth conformal Killing vector fields $X$ on $(N, g)$ are Killing vector fields on $(M, h)$.
\end{thm}
By Theorem \ref{thm_isometry_rigidity}, we choose $\hat{g}$ the smooth metric in the conformal infinity so that $Y$ is a Killing vector field on $(\partial M, \hat{g})$. Let $x$ be the corresponding geodesic defining function. Then $a_0(y)=0$ on $\partial M$. Therefore,
\begin{align*}
L_Zg(\frac{\partial}{\partial x^i}, \frac{\partial}{\partial x^j})
=\,\big[\,\frac{\partial b^k(0,y)}{\partial x^i} + b^s(0,y)\,\Gamma_{is}^k(g_x)\,\big]\frac{h_{kj}}{x^2}
+\,\big[\,\frac{\partial b^k(0,y)}{\partial x^j}+\, b^s(0,y)\,\Gamma_{js}^k(g_x)\,\big]\frac{h_{ki}}{x^2}.
\end{align*}
Direct calculation leads to
\begin{align*}
&L_Y\Gamma_{ij}^k(\hat{g})=\nabla_i^{\hat{g}}\nabla^{\hat{g}}_jY^k+R^k_{jim}(\hat{g})Y^m=0,\\
&L_YR_{\hat{g}}=0,\,\,L_YR_{ij}(\hat{g})=0,\\
&L_YR^i_{jkl}(\hat{g})=0,
\end{align*}
and similarly Lie derivatives $L_Y$ of covariant derivatives $\nabla_{\hat{g}}^{(k)}$ on the intrinsic curvature tensors of any order $k$ vanish on $(\partial M, \hat{g})$, see \cite{Yano1}. Therefore,
\begin{align*}
L_Zg(\frac{\partial}{\partial x^i}, \frac{\partial}{\partial x^j})=O(x),
\end{align*}
and
\begin{align*}
&L_Yg_{ij}^{(2k)}=0,\,\,\,\text{for}\,\,\,2\leq 2k\leq n-1,\\
&L_Y\tilde{g}_{ij}=0,
\end{align*}
since the coefficients $g^{(2k)}$ and $\tilde{g}$ in the expansion $(\ref{equn_expansion1})$ and $(\ref{equn_expansion2})$ are covariant derivative terms of the intrinsic curvature tensors in $(\partial M, \hat{g})$ for $2\leq 2k \leq n-1$.

Using the expression
\begin{align*}
\Gamma_{is}^k(g_x)=\frac{1}{2}h^{kp}(\frac{\partial}{\partial x^i}h_{ps}+\frac{\partial}{\partial x^s}h_{ip}-\frac{\partial}{\partial x^p}h_{is}),
\end{align*}
 the expansion $(\ref{equn_expansion1})$, $(\ref{equn_expansion2})$ with $g_{ij}^{(0)}=\hat{g}_{ij}$, and the expansion of the inverse matrices $h^{ij}$ of $h_{ij}$
\begin{align*}
&h^{ij}=\bar{g}_{(0)}^{ij}+x^2\bar{g}_{(2)}^{ij}+(\text{even terms})+x^{n-1}\bar{g}_{(n-1)}^{ij}+O(x^n),\,\,\,\text{for}\,\,n\,\,\text{odd},\\
&h^{ij}=\bar{g}_{(0)}^{ij}+x^2\bar{g}_{(2)}^{ij}+(\text{even terms})+x^{n-2}\bar{g}_{(n-2)}^{ij}+\bar{\tilde{g}}^{ij}x^n\log(x)+O(x^n),\,\,\,\text{for}\,\,n\,\,\text{even},
\end{align*}
it is easy to calculate that  for $n\geq 3$,
\begin{align*}
L_Zg(\frac{\partial}{\partial x^i}, \frac{\partial}{\partial x^j})&=\,\sum_{1 \leq k <\frac{n}{2}}x^{2k} [L_Yg_{ij}^{(2k)}+Y^lT_{lij}(k,n)]+O(x^n)\\
&=\sum_{1 \leq k <\frac{n}{2}}x^{2k}\,Y^lT_{lij}(k,n)+O(x^n),
\end{align*}
for $n$ odd, while
\begin{align*}
L_Zg(\frac{\partial}{\partial x^i}, \frac{\partial}{\partial x^j})&=\, \sum_{1 \leq k <\frac{n}{2}}x^{2k}[L_Yg_{ij}^{(2k)}+Y^lT_{lij}(k,n)]+x^n\log(x)\, L_Y\tilde{g}_{ij}+O(x^n)\\
&=\sum_{1 \leq k <\frac{n}{2}}x^{2k}\,Y^lT_{lij}(k,n)+O(x^n),
\end{align*}
for $n$ even, where $T_{lij}(k,n)$ are $(0,3)$-tensors consisted by covariant derivatives of curvature tensors of $\hat{g}$ of order up to $2k-1$, which are determined by the coefficients of lower powers of $x$ in the expansion of $g_x$.
Note that by direct calculations, for $2\leq 2k\leq n-1$,
\begin{align*}
&T_{lij}(k,n)=g_{jp}^{(2k)}\Gamma_{il}^p(\hat{g})+g_{ip}^{(2k)}\Gamma_{jl}^p(\hat{g})+\hat{g}_{jm}\bar{g}_{(2k)}^{mq}\hat{g}_{pq}\Gamma_{il}^p(\hat{g})+\hat{g}_{im}\bar{g}_{(2k)}^{mq}\hat{g}_{pq}\Gamma_{jl}^p(\hat{g})\\
&+\frac{1}{2}\sum_{\substack{a+b+c=k,\\
k-1\geq a,b,c\geq 0}}[g^{(2c)}_{jm}\bar{g}_{(2a)}^{mq}(\frac{\partial}{\partial x^i}g_{ql}^{(2b)}+\frac{\partial}{\partial x^l }g_{iq}^{(2b)}-\frac{\partial}{\partial x^q }g_{il}^{(2b)})+g^{(2c)}_{im}\bar{g}_{(2a)}^{mq}(\frac{\partial}{\partial x^j}g_{ql}^{(2b)}+\frac{\partial}{\partial x^l }g_{jq}^{(2b)}-\frac{\partial}{\partial x^q }g_{jl}^{(2b)})]\\
&=\sum_{b=0}^{k-1}[A_j^q(b, k)(\frac{\partial}{\partial x^i}g_{ql}^{(2b)}+\frac{\partial}{\partial x^l }g_{iq}^{(2b)}-\frac{\partial}{\partial x^q }g_{il}^{(2b)})+A_i^q(b,k)(\frac{\partial}{\partial x^j}g_{ql}^{(2b)}+\frac{\partial}{\partial x^l }g_{jq}^{(2b)}-\frac{\partial}{\partial x^q }g_{jl}^{(2b)})],
\end{align*}
where
\begin{align*}
&A_j^q(b,k)=\sum_{\substack{a+c=k-b,\\
k-1\geq a,c\geq 0}}g^{(2c)}_{jm}\bar{g}_{(2a)}^{mq}=0,\\
&A_i^q(b,k)=\sum_{\substack{a+c=k-b,\\
k-1\geq a,c\geq 0}}g^{(2c)}_{im}\bar{g}_{(2a)}^{mq}=0,
\end{align*}
by the definition of $\bar{g}_{(2a)}^{mq}$. Therefore,
\begin{align*}
T_{lij}(k,n)=0,
\end{align*}
for $2\leq 2k \leq n-1$ so that
\begin{align*}
L_Zg(\frac{\partial}{\partial x^i}, \frac{\partial}{\partial x^j})=O(x^{n-2}),
\end{align*}
for $1\leq i,j\leq n$, for any $n\geq 3$. Therefore, by Theorem \ref{thm_baseVector}, we obtain the following existence theorem of a Killing vector field on $(M,g)$.

\begin{thm}\label{thm_symmetryrealization}
 Let $\overline{M}$ be an $(n+1)$-dimensional compact smooth manifold with boundary $\partial M$ and $n\geq 3$. Assume that $(M, g)$ is a smooth conformally compact Einstein manifold with non-positive sectional curvature and a smooth conformal infinity $(\partial M, [\hat{g}])$. Then a conformal Killing vector field $Y$ extends to a vector field $X$ on $\overline{M}$ of $C^{n+1}$ up to $\partial M$ so that $X$ is a Killing vector field on $(M, g)$.
\end{thm}
Note that by the natural homeomorphism map $Exp^{+\infty}_{p_0}: UT_{p_0}M \to \partial M$ from the unit tangent sphere $UT_{p_0}M$ at $p_0$ to the infinity $\partial M$ and Lemma \ref{lem_vectorposition} or Lemma \ref{lem_stableset}, such an extension is unique. By Lemma \ref{lem_stableset} and Theorem \ref{thm_symmetryrealization}, we have
\begin{cor}
For any smooth metric $h$ on $S^n$ which is close enough to the round metric $g_0$ in $C^{2,\alpha}$, for any conformal Killing vector field $Y$ on $(M, h)$, the zero set of $Y$ must be the intersection of $\partial M$ and the closure of the submanifold $Exp_{p_0}(S)$ of $C^{n+1}$ in the closed Euclidean ball $\bar{B}_1(0)$, where $Exp_{p_0}$ is the exponential map at the spherical center of gravity $p_0$ of the conformally compact Einstein manifold $(B_1(0), g)$ with non-positive curvature, solved in \cite{GL}. In particular, if the zero set of the conformal Killing vector field $Y$ is non-empty, then it is either a two points set or a connected sub-manifold of dimension no less than one on $S^n$. Moreover, by Obata's Theorem and Theorem 34 in \cite{petersen}, The zero set is totally geodesic and of even codimension on $S^n$ under some metric $\tilde{h}\in[h]$.
\end{cor}

For any fixed $q=Exp^{+\infty}_{p_0}(v)\in \partial M$ corresponding to a unit vector $v\in UT_{p_0}M$ under the natural homeomorphism map $Exp^{+\infty}_{p_0}$, let $Y_1,...,Y_k$ be $k$ smooth conformal Killing vector fields on $(\partial M, [\hat{g}])$ which are linearly independent in $T_q\partial M$. They extend to the Killing vector fields $X_1,...,\,X_k$ in $(M, g)$, which are linearly independent at $Exp_{p_0}(tv)$ for $t>0$. Now assume $Y_1,...,Y_{n+k}$ are $(n+k)$ conformal Killing vector fields on $\partial M$ which form a linear basis of $T_q\partial M$ at each point $q \in \partial M$. Then they extend to $(n+k)$ Killing vector fields $X_1,...,\,X_{n+k}$ in $(M, g)$,
which form a linear basis of $T_pS_{p_0}(r)$ for $r>0$, with $p\in S_{p_0}(r)$ and $S_{p_0}(r)$ the $r$-geodesic sphere centered at $p_0$. By Theorem \ref{thm_isometry_rigidity}, we choose a smooth representation $\hat{g}$ of the conformal infinity under which $Y_1,...,\,Y_{n+k}$ are Killing vector fields. Let $x$ be the smooth geodesic defining function for this special metric $\hat{g}$. We will show that $x=C e^{-r}$ for some constant $C>0$ where $r$ is the distance function in $(M, g)$ to the spherical center of gravity $p_0$ of $(M, g)$.
\begin{thm}\label{thm_definingfunction}
Let $\overline{M}$ be an $(n+1)$-dimensional simply connected compact smooth manifold with boundary $\partial M$ and $n\geq 3$. Assume that $(M, g)$ is a smooth conformally compact Einstein manifold with non-positive sectional curvature and a smooth conformal infinity $(\partial M, [\hat{g}])$. Assume $p_0$ is the spherical center of gravity of $(M, g)$ and $r$ is the distance function to $p_0$ in $(M, g)$. Let $Y_1,...,Y_{n+k}$ be $(n+k)$ conformal Killing vector fields on $\partial M$ which form a linear basis of $T_q\partial M$ at each point $q \in \partial M$. Assume $\hat{g}$ is the representation in the conformal infinity under which $Y_1,...,Y_{n+k}$ are Killing vector fields. Let $x$ be the geodesic defining function for $\hat{g}$. Then $C^{\frac{1}{2}}x=e^{-r}$, for some constant $C>0$ and $0\leq x \leq C^{-\frac{1}{2}}$. In another word, let
\begin{align*}
g=dr^2+g_r.
\end{align*}
Then there exists a constant $C>0$ so that $\displaystyle\lim_{r\to+\infty}e^{-2r}g_r=C \hat{g}$ on $\partial M$.
\end{thm}
\begin{proof}
Let $X_1,...,X_{n+k}$ be the Killing vector fields in $(M, g)$ extended by $Y_1,...,Y_{n+k}$, which are of $C^{n+1}$ up to the boundary $\partial M$. The $1$-parameter Lie groups of $Y_1,...,Y_{n+k}$ generate a subgroup of the isometry on $\partial M$ under the action of which the orbit of each point on $\partial M$ covers $\partial M$. The same holds for $X_1,...,X_{n+k}$ on the geodesic sphere $S_{p_0}(r)$ centered at $p_0$ of radius $r>0$.

For any fixed $q\in \partial M$, let $(x^1,...,x^n)$ be a local coordinate in a neighborhood of $q$. Extending the functions $x^1,...,x^n$ in a small neighborhood $U$ of $q$ in $\overline{M}$ so that they are constants on each integral curve of $\frac{\partial}{\partial x}$ starting from the point on $\partial M$. We choose $(x^0, x^1,..., x^n)=(x, x^1,..., x^n)$ to be the local coordinate in $U$. Then the metric $g$ has the expression
\begin{align*}
g=x^{-2}\bar{g}=x^{-2}(dx^2+g_x).
\end{align*}
For $r$ large, by the argument in \cite{DJ} and \cite{LiQingShi}, $S_{p_0}(r)$ is a graph on $\partial M$ and moreover, the angle between the normal vector $\frac{\partial}{\partial r}$ of $S_{p_0}(r)$ and $-\frac{\partial}{\partial x}$ at points on $S_{p_0}(r)$ goes to zero uniformly as $r \to +\infty$. Without loss of generality, assume that $Y_1,...,Y_n$ are linearly independent at $q$. By continuity, we can choose the neighborhood $U$ small so that $X_1,..., X_n$ are linearly independent at each point in $U$. Consider $S_{p_0}(r)\bigcap U$ as the integral sub-manifold of the linear space generated by $X_1,...,X_n$. By the regularity of $X_1,...,X_n$, $\{S_{p_0}(r), r\,\,\text{large}\}$ is a family of graphs of at least $C^{n+1}$ on $\partial M$ in $(U, (x, x^1,...,x^n))$. 

Let $\vec{n}$ be the unit inner normal vector of $S_{p_0}(r)$ in $(\overline{M}, \bar{g})$ with $\partial M=S_{p_0}(+\infty)$, which is orthogonal to $X_1,...,X_n$ with regularity at least of $C^{n+1}$ in $U$. Denote $\gamma_{q}=\gamma_q(x)$ to be the integral curve of $\vec{n}$ starting from the fixed point $q \in \partial M$. Now we define a function $s=\varphi$ on $V$ so that along $\gamma_q$ the function $\varphi$ equals to the length function of $\gamma_q$ starting from the chosen point $q\in \partial M$, $\varphi=0$ on $\partial M$ and assume each $S_{p_0}(r)$ to be a level set of $\varphi$ in $U$ for $r>0$ large. It is clear that on each $S_{p_0}(r)$ in $U$, $\varphi$ equals to the length of $\gamma_q$ starting from $q$ to the intersection of $\gamma_q$ and $S_{p_0}(r)$ in $(\overline{M},\bar{g})$. By regularity of $X_1,..., X_n$, $\varphi$ is of $C^{n+1}$ in $U$. Since $r$ and $\varphi$ share the same level sets on $U\setminus \partial M$, we consider $r=f(\varphi)=f(s)$ with the function $f\in C^{n+1}$ on $U \setminus \partial M$ by smoothness of $r$ and the Killing vector fields in $M$. Therefore,
\begin{align*}
\nabla^gr=f'(s)\nabla^gs.
\end{align*}
Restricted on $\gamma_{q}$,
\begin{align*}
&<\frac{\partial}{\partial r}, \frac{\partial}{\partial r}>_g=1,\\
&<\frac{\partial}{\partial s}, \frac{\partial}{\partial s}>_{x^2g}=1,
\end{align*}
so that
\begin{align*}
x\frac{\partial}{\partial s}r=-1
\end{align*}
when restricted on $\gamma_q \subseteq U$. By the regularity of $\vec{n}$, we have the expansion
\begin{align*}
x=s+C_1s^2+C_2s^3+o(s^3)
\end{align*}
with some constants $C_1,\,C_2$, along $\gamma_q$. Let $\rho=e^{-r}$. Then when restricted on $\gamma_q$,
\begin{align*}
&\frac{d r}{ds}=-\frac{1}{s+C_1s^2+C_2s^3+o(s^3)},\\
&\rho=e^{-r_0+\int_{s_0}^s\frac{1}{s+C_1s^2+C_2s^3+o(s^3)}ds}=se^{F(s)},
\end{align*}
for some $F(s) \in C^2([0,s_0))$ with some $r_0>0$ and $s_0>0$ small. Therefore, $\rho=e^{-r}$ is at least of $C^2$ in a small neighborhood of $q$ in $\overline{M}$ for the fixed point $q\in \partial M$. 
By the arbitrary choice of $q\in \partial M$, $\rho=e^{-r}$ is of $C^2$ in a neighborhood of $\partial M$. 
Therefore,
\begin{align*}
\rho^2g=\frac{\rho^2}{x^2}\,x^2g
\end{align*}
can extend to a $C^2$ metric up to $\partial M$. Moreover, by continuity $X_1,..., X_{n+k}$ are still Killing vector fields of $\rho^2g$ up to the boundary. Therefore, the restriction of $\rho^2g$ on $\partial M$ is a homogeneous metric in $[\hat{g}]$ which is $C\hat{g}$ for some constant $C>0$. 
In summary, $\rho$ is of $C^2$ up to the boundary and it satisfies
\begin{align}
&\rho=0,\,\,\text{on}\,\,\partial M,\,\,\,\rho>0,\,\,\text{in}\,\,\,M,\\
&\label{equn_definingfunction}|d \rho|_{\rho^2g}=|dr|_{g}=1,\,\,\,\text{in a neighborhood of}\,\,\partial M,\\
&\lim_{\rho\to 0}(\rho^2g)\big|_{\partial M}=\,C\,\lim_{x\to 0}(x^2g)\big|_{\partial M}.
\end{align}
As discussed in \cite{Graham}, this is an initial value problem of the first order nonlinear differential equation $(\ref{equn_definingfunction})$ which is non-degenerate with $\partial M$ non-characteristic. There exists a unique $C^2$ solution to this initial value problem. But both $C^{\frac{1}{2}}x$ and $\rho$ are solutions to this problem. Therefore
\begin{align*}
C^{\frac{1}{2}}x=\rho=e^{-r},
\end{align*}
for $0 \leq x \leq C^{-\frac{1}{2}}$.

To end the proof, we have that the coordinates charts $(x,x^1,...,x^n)$ form a smooth coordinate cover on $\overline{M}\setminus \{p_0\}$. For convenience, we denote $\hat{g}=\displaystyle\lim_{r\to+\infty}(e^{-2r}g)\big|_{\partial M}$ and still denote $x$ to be the corresponding geodesic defining function of $\hat{g}$. Therefore, $0\leq x\leq 1$ on $\overline{M}$. We can pick up $(x, \theta^1,...,\theta^n)$ to be coordinate charts on $\overline{M}\setminus \{p_0\}$ so that the Killing vector $X_m=X_m^i(\theta)\frac{\partial}{\partial \theta^i}$ must be smooth up to the boundary for $1\leq m\leq n+k$, where $(r, \theta^1,...,\theta^n)$ is the polar coordinates on $(M, g)$ under the exponential map at $p_0$.
\end{proof}

\begin{Remark}
It is interesting to know how weak the regularity of the solutions is required in order that the unique existence of solutions to the initial value problem of the non-linear equation $(\ref{equn_definingfunction})$ still holds. If it is true for general compactification $\bar{g}=e^{-2r}g$ in Lipschitz sense and weak $W^{2,p}$ sense for any $p>1$ large with the representation $\hat{g}=\bar{g}-dx^2$ of the conformal infinity of certain regularity, then there is a possibility that the conclusion of Theorem \ref{thm_definingfunction} still holds on such conformally compact Einstein manifolds.
\end{Remark}

\section{Conformally Compact Einstein Manifolds with Homogeneous Conformal Infinity}\label{section_CCEinsteinHCITY}

Now we consider homogeneous metrics on the sphere $S^n$.
\begin{defn}
Let $(M, g)$ be a smooth Riemannian manifold. $M$ is called a homogeneous space if there is a (Lie) group $G$ of self-diffeomorphism of $(M,g)$ which acts
transitively. That is, for any $p,q \in M$, there exists $g\in G$ so that $p=gq$. If moreover, the action of each element on $G$ gives an isometry transformation on $M$, we call $(M, g)$ a Riemannian homogeneous space. $M$ is then diffeomorphic to $G/H$, where $H$ is the isotropy group of some point in $M$.
\end{defn}
In this paper, for a homogeneous space we always mean a Riemannian homogeneous space. Under the assumption in Theorem \ref{thm_definingfunction}, $\hat{g}=\displaystyle\lim_{r\to+\infty}(e^{-2r}g)\big|_{\partial M}$ is a homogeneous metric on $\partial M$ and $x=e^{-r}$ is the corresponding smooth geodesic defining function with $0\leq x\leq 1$ on $\overline{M}$. Fix a point $q\in \partial M$, with the coordinate $(0, \theta_0)=(0, \theta_0^1,...,\theta_0^n)$ under $(x,\theta^1,...,\theta^n)$. Without loss of generality, assume the Killing vector fields $Y_1,..., Y_n$ are linearly independent in a neighborhood of $\theta_0$ on $S^n$. Note that the extended Killing vector field $X_k=\displaystyle\sum_{m=1}^nX_k^m\frac{\partial}{\partial \theta^m}=\sum_{m=1}^nY_k^m\frac{\partial}{\partial \theta^m}$ with $X_k^m$ independent of $x$. Under the polar coordinates $(r, \theta)=(r, \theta^1,...,\theta^n)$,
\begin{align*}
g =dr^2+g_r = dr^2 + \sum_{i,j=1}^ng_{ij}d\theta^id \theta^j.
\end{align*}
Then
\begin{align*}
&\frac{\partial}{\partial \theta^i}(X_k^mg_{mj})+\frac{\partial}{\partial \theta^j}(X_k^mg_{mi})-2\Gamma_{ij}^p(g)X_k^mg_{mp}=0,\,\,\,\text{which is}\\
&\frac{\partial}{\partial \theta^i}X_k^pg_{pj}+\frac{\partial}{\partial \theta^j }X_k^pg_{pi}+X_k^q\frac{\partial}{\partial \theta^q}g_{ij}=0.
\end{align*}
Define the inverse of the matrix with elements $X_k^j$ for $1\leq k, j \leq n$ as
\begin{align*}
\left(\begin{matrix}&&&\\ &Z_i^j&&\\ &&& \end{matrix}\right)=\left(\begin{matrix}&&&\\ &X_i^j&&\\ &&& \end{matrix}\right)^{-1}.
\end{align*}
We denote
\begin{align*}
C_{ij}^p=Z_i^k\frac{\partial}{\partial \theta^j}X_k^p,
\end{align*}
and
\begin{align}\label{equn_twotensor}
T_{ij}^p = - T_{ji}^p=C_{ij}^p- C_{ji}^p=Z_i^kZ_j^m[X_m, X_k]^p,
\end{align}
with $[X_m, X_k]$ Lie bracket of the vectors $X_m$ and $X_k$.
Note that $C_{ij}^p$ and $T_{ij}^p$ are independent of $r$ and the metric. Then
\begin{align}
&\label{equn_tangentspace}\frac{\partial}{\partial \theta^q}g_{ij}=-C_{qi}^mg_{mj}-C_{qj}^mg_{mi},\\
&\Gamma_{ij}^p(g_r)=\frac{1}{2}[-(C_{ij}^p+C_{ji}^p)+g^{pq}(-C_{iq}^mg_{mj}-C_{jq}^mg_{mi}+C_{qi}^mg_{mj}+C_{qj}^mg_{mi})].
\end{align}
By the differential equations $(\ref{equn_tangentspace})$, there exists $E_i^j(\theta)$ such that
\begin{align*}
g_{ij}(r, \theta)=E_i^p(\theta)g_{pq}(r, \theta_0)E_j^q(\theta),
\end{align*}
for $\theta$ in a neighborhood of $\theta_0$. Direct calculation yields, that for $i,j\geq 1$ the Ricci tensor is
\begin{align}
R_{ij}(g_r)&\label{equn_Riccitensor}=\frac{1}{2}(\frac{\partial}{\partial \theta^p}T_{ij}^p+C_{ip}^kT_{kj}^p+C_{kj}^pT_{ip}^k+C_{kp}^pT_{ji}^k)-\frac{1}{2}g^{pq}(\frac{\partial}{\partial \theta^p}T_{iq}^m+C_{ip}^kT_{kq}^m+C_{pq}^kT_{ik}^m-C_{pk}^mT_{iq}^k)g_{mj}\\
&-\frac{1}{2}g^{pq}(\frac{\partial}{\partial \theta^p}T_{jq}^m+C_{jp}^kT_{kq}^m+C_{pq}^kT_{jk}^m-C_{pk}^mT_{jq}^k)g_{mi}+\frac{1}{4}T_{pi}^kT_{kj}^p-\frac{1}{4}g^{pq}T_{pi}^kT_{kq}^mg_{mj}\notag\\
&-\frac{1}{4}g^{pq}T_{pj}^kT_{kq}^mg_{mi}-\frac{1}{2}g^{kq}T_{pk}^pT_{iq}^mg_{mj}-\frac{1}{2}g^{kq}T_{pk}^pT_{jq}^mg_{mi}+\frac{1}{4}g^{pq}T_{pj}^kT_{iq}^mg_{km}+\frac{1}{4}g^{pq}T_{pi}^kT_{jq}^mg_{km}\notag\\
&-\frac{1}{4}g^{pl}(T_{jl}^mg_{mk}+T_{kl}^mg_{mj})g^{kq}(T_{pq}^mg_{mi}+ T_{iq}^mg_{mp})\notag.
\end{align}
Note that the Einstein equations are
\begin{align*}
R_{00}(g)=-n,\,\,R_{0i}(g)=0,\,\,R_{ij}(g)=-ng_{ij},
\end{align*}
for $i,j\geq 1$. Under local coordinates, the system becomes
\begin{align*}
&g^{pq}\frac{\partial^2}{\partial r^2}g_{pq}-\frac{1}{2}g^{pk}\frac{\partial}{\partial r}g_{km}g^{mq}\frac{\partial}{\partial r}g_{qp}=2n,\\ 
&C_{pq}^pg^{qk}\frac{\partial}{\partial r}g_{ki}+C_{iq}^mg^{pq}\frac{\partial}{\partial r}g_{mp}-C_{qp}^pg^{qk}\frac{\partial}{\partial r}g_{ki}-C_{pi}^kg^{pq}\frac{\partial}{\partial r}g_{kq}=0,\,\,\,\,i\geq 1,\\
&\frac{\partial^2}{\partial r^2}g_{ij}-2R_{ij}(g_r)+\frac{1}{2}g^{pq}\frac{\partial}{\partial r}g_{pq}\frac{\partial}{\partial r}g_{ij}-g^{pq}\frac{\partial}{\partial r}g_{pi}\frac{\partial}{\partial r}g_{qj}=2ng_{ij},\,\,\,\,i,j\geq 1.
\end{align*}

Since $g_r$ is a homogeneous metric on $S_{p_0}(r)$, we only need to consider the Einstein equations at $(r, \theta_0)$ for $r>0$. Let
\begin{align*}
g_r=\sinh^2(r)\bar{h}=\frac{x^{-2}(1-x^2)^2}{4}\bar{h}_{ij}d\theta^id\theta^j.
\end{align*}
At $\theta=\theta_0$, the Einstein equations becomes
\begin{align}
&\label{equn_EinsteinODEs1}\frac{d}{d x}(x(1-x^2)\bar{h}^{pq}\frac{\partial}{\partial x}\bar{h}_{pq})+\frac{1}{2}x(1-x^2)\bar{h}^{pk}\frac{d}{d x}\bar{h}_{km}\bar{h}^{mq}\frac{d}{d x}\bar{h}_{qp}-2\bar{h}^{pq}\frac{d}{d x} \bar{h}_{pq}= 0,\\ 
&\label{equn_EinsteinODEs2}C_{pq}^p\bar{h}^{qk}\frac{d}{d x}\bar{h}_{ki}+C_{iq}^m\bar{h}^{pq}\frac{d}{d x}\bar{h}_{mp}-C_{qp}^p\bar{h}^{qk}\frac{d}{d x}\bar{h}_{ki}-C_{pi}^k\bar{h}^{pq}\frac{d}{d x}\bar{h}_{kq}=0,\\ \notag \\
&\label{equn_EinsteinODEs3}-\frac{1}{8}x(1-x^2)^2\frac{d^2}{d x^2}\bar{h}_{ij}+\frac{1}{8}[(n-1)+(1+n)x^2]\,(1-x^2)\frac{d}{d x}\bar{h}_{ij}+\frac{x(1-x^2)^2}{8}\bar{h}^{pq}\frac{d}{d x}\bar{h}_{pi}\frac{d}{d x}\bar{h}_{qj}\\
&+\frac{1}{8}(1+x^2)(1-x^2)\bar{h}^{pq}\frac{d}{d x}\bar{h}_{pq}\bar{h}_{ij}-\frac{1}{16}x(1-x^2)^2\bar{h}^{pq}\frac{d}{d x}\bar{h}_{pq}\frac{d}{d x}\bar{h}_{ij}+(1-n)x\bar{h}_{ij}+xR_{ij}(\bar{h})=0,\notag
\end{align}
for $1\leq i,j\leq n$, with $C_{ij}^p(\theta_0)$ independent of $x$, and $R_{ij}(\bar{h})=R_{ij}(g_r)$ with formula $(\ref{equn_Riccitensor})$. This is a system of ordinary differential equations for $\bar{h}_{ij}$ on $x\in [0,1]$, with the boundary conditions
\begin{align}\label{equn_boundaryvalue1}
\bar{h}_{ij}(0)\in [\hat{g}_{ij}(\theta_0)],\,\,\bar{h}_{ij}(1)=g^0_{ij},\,\,\frac{d}{d x}\bar{h}_{ij}(0)=0,\,\,\frac{d}{d x}\bar{h}_{ij}(1)=0,
\end{align}
with $g^0$ the round metric on $S^n$. Note that the homogeneous metric in $[\hat{g}]$ is unique up to a constant multiplier, and then the initial data of $\bar{h}_{ij}(0)$ is determined up to a constant multiplier. Since the metric is left invariant on each $S_{p_0}(r)$, once $\bar{h}_{ij}(x)=\bar{h}_{ij}(x,\theta_0)$ is determined on $0\leq x\leq 1$, the Einstein metric
\begin{align*}
g=dr^2+g_r=x^{-2}(dx^2+\frac{(1-x^2)^2}{4}\bar{h})
\end{align*}
is determined completely. Therefore, uniqueness of the non-positively curved conformally compact Einstein metric, with the prescribed conformal infinity with a homogeneous representation, is equivalent to uniqueness of the solution $\bar{h}_{ij}=\bar{h}_{ij}(x)$ to the system $(\ref{equn_EinsteinODEs1})-(\ref{equn_EinsteinODEs3})$ with the boundary data $(\ref{equn_boundaryvalue1})$ so that $\bar{h}_{ij} \in C^{n-1}([0,1])\bigcap C^{\infty}((0,1])$. We will simplify the system of ordinary differential equations when $(\partial M, \hat{g})$ is a generalized Berger sphere and consider uniqueness of the solution to this boundary value problem. Higher dimensional case will be discussed else where.

There is a nice description on the homogeneous metrics on $S^n$ in \cite{Ziller1}. On the sphere $S^3$, we identify $S^3$ and the Lie group $\text{SU}(2)$ by the map
\begin{align*}
(z, w)\in S^3 \subseteq \mathbb{C}^2 \,\mapsto\,\,\left(\begin{matrix}z&w&\\ -\bar{w}&\bar{z}& \end{matrix}\right)\in \text{SU}(2).
\end{align*}
Let $\{Y_1,\,Y_2,\,Y_3\}$ be the basis of the Lie algebra $su(2)$ of $\text{SU}(2)$ given by
\begin{align*}
Y_1=\left(\begin{matrix}i&0&\\ 0&-i& \end{matrix}\right),\,Y_2=\left(\begin{matrix}0&1&\\ -1& 0& \end{matrix}\right),\,\,Y_3 =\left(\begin{matrix}0&i&\\ i&0& \end{matrix}\right),
\end{align*}
so that
\begin{align}\label{equn_symmetry1}
[Y_i,Y_j]=2\varepsilon_{ijk}Y_k,\,\,\,i,j,k=1,2,3,
\end{align}
with $\varepsilon_{ijk}=1$ when $(i-j)(j-k)(k-i)>0$; $\varepsilon_{ijk}=-1$ when $(i-j)(j-k)(k-i)<0$; and $\varepsilon_{ijk}=0$ otherwise. Note that for a triple $(\tilde{Y}_1, \tilde{Y}_2, \tilde{Y}_3)$ in $su(2)$, the identities $(\ref{equn_symmetry1})$ still hold if and only if $(\tilde{Y}_1, \tilde{Y}_2, \tilde{Y}_3)=(Y_1, Y_2, Y_3)C$ with $C\in \text{SO}(3)$. Consider the Riemannian manifold $(S^3, \hat{g})$ in which $Y_1,\,Y_2,\,Y_3$ are three linearly independent right invariant Killing vector fields, with three corresponding left invariant $1$-forms $\sigma_1,\,\sigma_2,\,\sigma_3$ and $\hat{g}$ is a left invariant metric. We can choose the triple $(Y_1,Y_2,Y_3)$ so that $(\ref{equn_symmetry1})$ holds and $\hat{g}$ has the form
\begin{align*}
\hat{g}=\lambda_1\sigma_1^2+\lambda_2\sigma_2^2+\lambda_3\sigma_3^2,
\end{align*}
with $\lambda_1,\,\lambda_2,\, \lambda_3$ three positive numbers. Such a metric $\hat{g}$ is called a {\em generalized Berger metric}. If two of the numbers coincide, for instance, $\lambda_2=\lambda_3$, we call $\hat{g}$ a {\em Berger metric} on $S^3$. For a Berger metric, there exists a fourth Killing vector field $Y_4$ on $S^3$ which is left invariant.

Let $\overline{M}=\overline{B}_1(0)$ be diffeomorphic to the closed unit ball in the Euclidean space $\mathbb{R}^4$, with boundary $S^3$. Assume that $(M, g)$ is a smooth conformally compact Einstein manifold with non-positive sectional curvature and a smooth conformal infinity $(S^3, [\hat{g}])$, where $(S^3, \hat{g})$ is a Berger sphere. Let $Y_1,\,Y_2,\,Y_3$ be three right invariant Killing vector fields on $(S^3, \hat{g})$ as above. As in Theorem \ref{thm_definingfunction}, we can choose $\hat{g}=\displaystyle\lim_{r\to+\infty}(e^{-2r}g)\big|_{\partial M}$ and $x=e^{-r}$ is the corresponding geodesic defining function. Then $Y_1,\,Y_2,\,Y_3$ extend to three Killing vector fields $X_1=X_1^k(\theta)\frac{\partial}{\partial \theta^k},\,X_2=X_2^k(\theta)\frac{\partial}{\partial \theta^k},\,X_3=X_3^k(\theta)\frac{\partial}{\partial \theta^k}$ in $(M, g)$ with $(x, \theta^1,\theta^2,\theta^3)$ coordinate charts on $M\setminus \{p_0\}$ and $p_0$ the spherical center of gravity. Let $\sigma_1,\,\sigma_2,\,\sigma_3$ be the corresponding left invariant $1$-forms on $S^3$.

Therefore, $X_1, X_2, X_3$ are the right invariant Killing vector fields on $S_{p_0}(r)$ for all $r>0$ and $(\ref{equn_symmetry1})$ holds for $X_1,\,X_2,\,X_3$. Let
\begin{align*}
g=x^{-2}(dx^2+g_x),
\end{align*}
where $x=e^{-r}$. Then for $0\leq x <1$,
\begin{align*}
g_x=\frac{(1-x^2)^2}{4}\bar{h}=\lambda_1(x)\tilde{\sigma}_1^2+\lambda_2(x)\tilde{\sigma}_2^2+\lambda_3(x)\tilde{\sigma}_3^2,
\end{align*}
with the coefficients $\lambda_1,\,\lambda_2,\,\lambda_3$ depending on $x$ and the left invariant $1$-forms $\tilde{\sigma}_1,\tilde{\sigma}_2,\tilde{\sigma}_3$ corresponding to the triple $(\tilde{X}_1,\tilde{X}_2,\tilde{X}_3)=(X_1, X_2, X_3)C$ on each $S_{p_0}(r)$, where $C\in \text{SO}(3)$ depends on $x$. To show the uniqueness of conformally compact Einstein metrics with the prescribed conformal infinity, we
only need to show uniqueness of the metrics $\bar{h}(x)=\bar{h}(x, \theta_0)$ at the fixed point $\theta_0\in S^3$ along $x\in[0,1]$. 
For convenience of calculation of the curvature terms in the Einstein equations, we will use local coordinates.

We now computer $C_{ij}^p$ in $(\ref{equn_Riccitensor})$. 
Consider $S^3$ as the unit sphere in $\mathbb{R}^4$ with the coordinate $(x,y,u,v)$. Then on $S^3$, $x^2+y^2+u^2+v^2=1$. Up to the natural $\text{SO}(3)$ action, the three right invariant Killing vector fields on generalized Berger sphere $S^3$ are
\begin{align*}
X_1=(-y,x,v,u),\,\,X_2=(u,v,-x,-y),\,\,X_2=(v,-u,y,-x).
\end{align*}
Without loss of generality, the fixed point $\theta_0$ we choose on $S^3$ is $(1,0,0,0)$ under this coordinate. In order to choose a local coordinate $(\theta^1, \theta^2, \theta^3)$ in a neighborhood of $\theta_0$ on $S^3$ so that $X_i(\theta_0)=\frac{\partial}{\partial \theta^i}$ for $i=1,2,3$, we choose the coordinate
\begin{align}\label{equn_base}
\theta^1=y,\,\theta^2=-u,\,\theta^3=-v,
\end{align}
in a neighborhood of $\theta_0$. Therefore,
\begin{align*}
&X_1=\sqrt{1-(\theta^1)^2-(\theta^2)^2-(\theta^3)^2}\frac{\partial}{\partial \theta^1}+\theta^3\frac{\partial}{\partial \theta^2}-\theta^2\frac{\partial}{\partial \theta^3},\\
&X_2=-\theta^3\frac{\partial}{\partial \theta^1}+\sqrt{1-(\theta^1)^2-(\theta^2)^2-(\theta^3)^2}\frac{\partial}{\partial \theta^2}+\theta^1\frac{\partial}{\partial \theta^3},\\
&X_3=\theta^2\frac{\partial}{\partial \theta^1}-\theta^1\frac{\partial}{\partial \theta^2}+\sqrt{1-(\theta^1)^2-(\theta^2)^2-(\theta^3)^2}\frac{\partial}{\partial \theta^3}.
\end{align*}
In particular, at $\theta_0$, the element $Z_i^j$ of inverse of the matrix $\big[X_i^j\big]$ has $Z_i^j=\delta_i^j$ under $(\theta^1, \theta^2, \theta^3)$. Therefore, at $\theta_0$, $C_{ij}^k= \varepsilon_{jik}$ and $T_{ij}^k=2\varepsilon_{jik}$ with $\varepsilon_{jik}$ defined as in $(\ref{equn_symmetry1})$. Note that $C_{ij}^k$ is independent of the metric but depends only on the relationship $(\ref{equn_symmetry1})$. By $(\ref{equn_twotensor})$ and $(\ref{equn_symmetry1})$, we have that
\begin{align*}
T_{ij}^p=2Z_i^kZ_j^bX_a^p\varepsilon_{bka}.
\end{align*}
Therefore,
\begin{align*}
\frac{\partial}{\partial \theta^m}T_{ij}^p&=\,-2 \varepsilon_{cba}(-Z_i^d\frac{\partial}{\partial \theta^m}X_d^q\,Z_q^cZ_j^bX_a^p-Z_i^cZ_j^d\frac{\partial}{\partial \theta^m}X_d^q\,Z_q^bX_a^p+Z_i^cZ_j^b\frac{\partial}{\partial \theta^m}X_a^p)\\
&=\,-2 \varepsilon_{cba}(-C_{im}^q\,Z_q^cZ_j^bZ_a^p-Z_i^cC_{jm}^q\,Z_q^bX_a^p+Z_i^cZ_j^bX_a^qC_{qm}^p)\\
&=\,2\varepsilon_{qjp}C_{im}^q+ 2 \varepsilon_{iqp}C_{jm}^q- 2\varepsilon_{ijq}C_{qm}^p\\
&=\,2\sum_{q=1}^3(\varepsilon_{qjp}\varepsilon_{miq}+ \varepsilon_{iqp}\varepsilon_{mjq}- \varepsilon_{ijq}\varepsilon_{mqp}).
\end{align*}
at the point $\theta_0\in S^3$. Substituting all these data to the expression $(\ref{equn_Riccitensor})$ of the Ricci curvature tensor of $\bar{h}$, we have that $R_{ij}(\bar{h})$ vanishes identically for $i\neq j$.
\begin{lem}\label{lem_symmetry_Threedimensional}
Under the polar coordinate $(x, \theta^1, \theta^2, \theta^3)$ with $(\theta^1, \theta^2, \theta^3)$ chosen in $(\ref{equn_base})$ and $x$ the geodesic defining function about $\hat{g}$, the metric satisfies
\begin{align*}
\bar{h}=I_1(x)d(\theta^1)^2+I_2(x)d(\theta^2)^2+I_3(x)d(\theta^3)^2,
\end{align*}
with some positive functions $I_1,\,I_2,\,I_3$ at the point $(x, \theta_0)$ for $0\leq x \leq1$ where $I_i(1)=1$, $i=1,2,3$.
\end{lem}
\begin{proof}
If $\hat{g}$ is a Berger metric, there exists a left invariant Killing vector field $Y_4$. Without loss of generality, let $Y_4=Y_1$ at $\theta_0$. Then
\begin{align}\label{equn_Bergermetric1}
\hat{g}=\lambda_1 d(\theta^1)^2\,+\,\lambda_2(d(\theta^2)^2\,+\,d(\theta^3)^2).
\end{align}
$Y_1,...,Y_4$  extend to four Killing vector fields $X_1,...,X_4$ in $(M, g)$, which guarantees that the metric $\bar{h}$ has the form
\begin{align*}
\bar{h}=I_1(x) d(\theta^1)^2+ I_2(x)(d(\theta^2)^2+d(\theta^3)^2)
\end{align*}
under the coordinate $(x, \theta^1,\theta^2,\theta^3)$.

We give one way to see this. In the coordinate $(x,y,u,v)\in S^3\subseteq \mathbb{R}^4$, the left invariant vector field $X_4=(-y,x,-v,u)$. Under the coordinate $(\ref{equn_base})$,
\begin{align*}
X_4=\sqrt{1-(\theta^1)^2-(\theta^2)^2-(\theta^3)^2}\frac{\partial}{\partial \theta^1}-\theta^3\frac{\partial}{\partial \theta^2}+\theta^2\frac{\partial}{\partial \theta^3}.
\end{align*}
With $X_1$ replaced by $\tilde{X}_1=X_4$ in the definition of $Z_i^j$ and $C_{ij}^p$ we obtain $\tilde{Z}_i^j$ and $\tilde{C}_{ij}^p$ for $1\leq i,j,k\leq 3$. And instead of $(\ref{equn_tangentspace})$, we have
\begin{align}\label{equn_tangentspace2}
\frac{\partial}{\partial \theta^q}g_{ij}=-\tilde{C}_{qi}^mg_{mj}-\tilde{C}_{qj}^mg_{mi}.
\end{align}
By direct calculation, at $\theta_0\in S^3$, which is $(x,y,u,v)=(1,0,0,0)$,
\begin{align*}
\tilde{C}_{1j}^p=-C_{1j}^p,\,\,\,\tilde{C}_{ij}^p =\,C_{ij}^p
\end{align*}
for $i=2,3$ and $1\leq j,p\leq 3$. By $(\ref{equn_tangentspace})$ and $(\ref{equn_tangentspace2})$, we have
\begin{align*}
-C_{qi}^mg_{mj}-C_{qj}^mg_{mi}=-\tilde{C}_{qi}^mg_{mj}-\tilde{C}_{qj}^mg_{mi}
\end{align*}
Take $q=1$. Let $i=2$ and $j=3$, we have that
\begin{align*}
g_{22}=g_{33}.
\end{align*}
For the data $(i,j)=(1,2)$, $(i,j)=(1,3)$ and $(i,j)=(2,2)$ we have
\begin{align*}
g_{12}=g_{13}=g_{23}=0.
\end{align*}
One can check that these are all the symmetry we have. Therefore, the lemma holds for the case when $\hat{g}$ is a Berger metric.

If $\hat{g}$ is a generalized Berger metric
\begin{align}\label{equn_GBergermetric}
\hat{g}=\lambda_1 d(\theta^1)^2\,+\,\lambda_2 d(\theta^2)^2\,+\,\lambda_3d(\theta^3)^2.
\end{align}
with $\lambda_1,\,\lambda_2$ and $\lambda_3$ different from each other, we will use the Einstein equations. By $(\ref{equn_EinsteinODEs2})$, we have that
\begin{align*}
&-\bar{h}^{13}\frac{d}{d x}\bar{h}_{12}+\bar{h}^{12}\frac{d}{dx}\bar{h}_{13}+(\bar{h}^{22}-\bar{h}^{33})\frac{d}{dx}\bar{h}_{23}=\bar{h}^{23}(\frac{d}{dx}\bar{h}_{22}-\frac{d}{dx}\bar{h}_{33}),\\
&\bar{h}^{23}\frac{d}{d x}\bar{h}_{12}+(\bar{h}^{33}-\bar{h}^{11})\frac{d}{dx}\bar{h}_{13}-\bar{h}^{12}\frac{d}{dx}\bar{h}_{23}=\bar{h}^{13}(\frac{d}{dx}\bar{h}_{33}-\frac{d}{dx}\bar{h}_{11}),\\
&(\bar{h}^{11}-\bar{h}^{22})\frac{d}{dx}\bar{h}_{12}-\bar{h}^{23}\frac{d}{d x}\bar{h}_{13}+\bar{h}^{13}\frac{d}{dx}\bar{h}_{23}=\bar{h}^{12}(\frac{d}{dx}\bar{h}_{11}-\frac{d}{dx}\bar{h}_{22}),
\end{align*}
with $x$ the geodesic defining function. We consider $\bar{h}$ as a known $C^2$ solution to the system with the initial data
\begin{align*}
\bar{h}(0)=\lambda_1 d(\theta^1)^2\,+\,\lambda_2 d(\theta^2)^2\,+\,\lambda_3d(\theta^3)^2.
\end{align*}
Consider the system as a system of linear equations of $(\frac{d}{d x}\bar{h}_{12}, \frac{d}{dx}\bar{h}_{13}, \frac{d}{dx}\bar{h}_{23})$. Then on the left hand side, the matrix of coefficients is invertible at $x=0$ by our assumption. By definition, for $i\neq j$, on the right hand side of the system the element function $\bar{h}^{ij}(x)$ of the inverse matrix of $\bar{h}$ can be expressed as
\begin{align*}
\bar{h}^{ij}=b_{ij}^1(x)\bar{h}_{12}+b_{ij}^2(x)\bar{h}_{13}+b_{ij}^3(x)\bar{h}_{23}
\end{align*}
with $b_{ij}^k(x)$ some function of $C^2([0,1))$. Now the system can be viewed as the system of ordinary differential equations:
\begin{align*}
A(x)\left(\begin{matrix}\frac{d\bar{h}_{12}}{d x}\\ \frac{d\bar{h}_{13}}{d x}\\ \frac{d\bar{h}_{23}}{d x}\end{matrix}\right)=B(x)\left(\begin{matrix}\bar{h}_{12}\\ \bar{h}_{13}\\ \bar{h}_{23}\end{matrix}\right)
\end{align*}
with the matrix $A(x)$ invertible in $x\in[0,\varepsilon)$ for some $\varepsilon>0$ small, and $A(x)$ and $B(x)$ are of $C^1$. Now this system is a system of linear homogeneous ordinary differential equations of $(\bar{h}_{12}, \bar{h}_{13}, \bar{h}_{23})$, with initial data $\bar{h}_{ij}(0)=0$ for $i\neq j$. Then by uniqueness of the solution to the initial value problem, $\bar{h}_{ij}(x)=0$ in $x\in[0,\varepsilon]$ for $i\neq j$. Since $R_{ij}(\bar{h})=0$ for $i\neq j$, similarly we consider the three equations in $(\ref{equn_EinsteinODEs3})$ when $i\neq j$ as a system of three second order linear homogeneous ordinary differential equations of $(\bar{h}_{12}, \bar{h}_{13}, \bar{h}_{23})$ in $x\in (0,1)$. Using the conclusion that $\bar{h}_{ij}(x)=0$ in $x\in[0,\varepsilon]$ for $i\neq j$, by uniqueness of the solution to the initial value problem, we have that $\bar{h}_{ij}(x)=0$ in $x\in[0, 1]$ for $i\neq j$. At the center of gravity $p_0$ of the manifold, $x=e^{-r}=1$. This proves the lemma.


\end{proof}
Substituting all these data to the expression $(\ref{equn_Riccitensor})$ of the Ricci curvature tensor of $\bar{h}$, we have that
\begin{align*}
&R_{11}(\bar{h})=4-2(I_3^{-1}I_2+I_2^{-1}I_3)+2I_2^{-1}I_3^{-1}I_1^2,\\
&R_{22}(\bar{h})=4-2(I_1^{-1}I_3+I_3^{-1}I_1)+2I_1^{-1}I_3^{-1}I_2^2,\\
&R_{33}(\bar{h})=4-2(I_1^{-1}I_2+I_2^{-1}I_1)+2I_1^{-1}I_2^{-1}I_3^2,
\end{align*}
Denote $K=\text{det}(\bar{h}_{ij})=I_1I_2I_3$.

Therefore, for a Berger metric $\hat{g}$, the system $(\ref{equn_EinsteinODEs1})-(\ref{equn_EinsteinODEs3})$ becomes
\begin{align*}
&\frac{d}{d x}[x(1-x^2)(I_1^{-1}\frac{d}{d x}I_1+2I_2^{-1}\frac{d}{d x}I_2)]+\frac{1}{2}x(1-x^2)[(I_1^{-1}\frac{d}{d x}I_1)^2+2(I_2^{-1}\frac{d}{d x}I_2)^2]\\
&-2(I_1^{-1}\frac{d}{d x}I_1+2I_2^{-1}\frac{d}{d x}I_2)= 0,\\
&-\frac{1}{8}x(1-x^2)^2I_1''+(\frac{1}{2}x^2+\frac{1}{4})(1-x^2)I_1'+\frac{1}{8}x(1-x^2)^2I_1^{-1}(I_1')^2\\ &+\frac{1}{8}(1+x^2)(1-x^2)(2I_2^{-1}I_2'+I_1^{-1}I_1')I_1 -\frac{1}{16}x(1-x^2)^2(2I_2^{-1}I_2'+I_1^{-1}I_1')I_1'-2xI_1+2xI_2^{-2}I_1^2=0,\\
&-\frac{1}{8}x(1-x^2)^2I_2''+(\frac{1}{2}x^2+\frac{1}{4})(1-x^2)I_2'+\frac{1}{8}(1+x^2)(1-x^2)(2I_2'+I_1^{-1}I_2I_1')\\
&-\frac{1}{16}x(1-x^2)^2I_1^{-1}I_2'I_1'-2xI_2+x(4-2I_2^{-1}I_1)=0,
\end{align*}
on $x\in [0,1]$ where $I_i'=\frac{d}{dx}I_i$, with the boundary condition
\begin{align}\label{equn_boundaryvaluediagn1}
\frac{I_1(0)}{I_2(0)}=\frac{\lambda_1}{\lambda_2},\,\,I_1(1)=I_2(1)=1,\,\,I_1'(0)=I_2'(0)=I_1'(1)=I_2'(1)=0.
\end{align}
Denote $\phi=\frac{I_2}{I_1}$, $y_1=\log(K)$ and $y_2=\log(\phi)$ so that
\begin{align*}
I_2=(K\phi)^{\frac{1}{3}},\,\,I_1=(K\phi^{-2})^{\frac{1}{3}}.
\end{align*}
Therefore, for the Berger metric $\hat{g}$ in $(\ref{equn_Bergermetric1})$, the boundary value problem of the Einstein metrics becomes
\begin{align}
&\label{equn_BergerEinstein01}y_1''+\frac{1}{6}(y_1')^2+\frac{1}{3}(y_2')^2-x^{-1}(1+3x^2)(1-x^2)^{-1}y_1'=0,\\
&\label{equn_BergerEinstein02}y_1''+\frac{1}{2}(y_1')^2-x^{-1}(5+7x^2)(1-x^2)^{-1}y_1'+8(1-x^2)^{-2}(6-8K^{-\frac{1}{3}}\phi^{-\frac{1}{3}}+2K^{-\frac{1}{3}}\phi^{-\frac{4}{3}})=0,\\
&\label{equn_BergerEinstein03}y_2''+\frac{1}{2}y_1'y_2'-2x^{-1}(1+2x^2)(1-x^2)^{-1}y_2'+32(1-x^2)^{-2}K^{-\frac{1}{3}}\phi^{-\frac{1}{3}}(\phi^{-1}-1)=0.
\end{align}
for $y_1(x),y_2(x)\in C^{\infty}([0,1])$ with the boundary condition
\begin{align}\label{equn_BergerBV01}
\phi(0)=\frac{\lambda_2}{\lambda_1},\,\,K(1)=\phi(1)=1,\,\,y_1'(0)=y_2'(0)=y_1'(1)=y_2'(1)=0.
\end{align}
Combining $(\ref{equn_BergerEinstein01})$ and $(\ref{equn_BergerEinstein02})$, we have
\begin{align}\label{equn_BergerEinstein04}
(y_1')^2-(y_2')^2-12x^{-1}(1+x^2)(1-x^2)^{-1}y_1'+48(1-x^2)^{-2}(3-4(K\phi)^{-\frac{1}{3}}+K^{-\frac{1}{3}}\phi^{-\frac{4}{3}})=0.
\end{align}
By $(\ref{equn_expansion1})$, we have the expansion of $y_1$ and $y_2$ at $x=0$, which can also be done directly using the system $(\ref{equn_BergerEinstein01})-(\ref{equn_BergerEinstein03})$ and the boundary data $(\ref{equn_BergerBV01})$. Let $\Phi(x)$ be the function on the left hand side of the equation $(\ref{equn_BergerEinstein04})$. Take derivative of $\Phi$ and use the equations $(\ref{equn_BergerEinstein02})$ and $(\ref{equn_BergerEinstein03})$ we have
\begin{align}\label{equn_Berger2-21}
\Phi'+(y_1'-4x^{-1}(1+2x^2)(1-x^2)^{-1})\Phi=0.
\end{align}
Consider $y_1'$ as a given function. By the expansion $(\ref{equn_expansion1})$, especially that $y_1'(0)=0$ and $\text{tr}_{\hat{g}}g^{(3)}=\frac{1}{4}(I_1(0)^{-1}I_1^{(3)}+2I_2(0)^{-1}I_2^{(3)})=0$, $(\ref{equn_Berger2-21})$ has a unique solution $\Phi=0$, which is $(\ref{equn_BergerEinstein04})$. Therefore, $(\ref{equn_BergerEinstein02})$ and $(\ref{equn_BergerEinstein03})$ combining with the expansion of the Einstein metric imply $(\ref{equn_BergerEinstein04})$. Similarly, any two of the equations $(\ref{equn_BergerEinstein01})-(\ref{equn_BergerEinstein04})$ combining with the boundary expansion of the Einstein metric give the other two equations. Note that the coefficients of the expansion of the metric can be solved inductively by the equations $(\ref{equn_BergerEinstein02})-(\ref{equn_BergerEinstein03})$ and the initial data $(\ref{equn_BergerBV01})$ before the order $x^3$.

For a generalized Berger metric $\hat{g}$ in $(\ref{equn_GBergermetric})$ where $\lambda_1, \lambda_2$ and $ \lambda_3$ differ from one another, the system $(\ref{equn_EinsteinODEs1})-(\ref{equn_EinsteinODEs3})$ becomes
\begin{align*}
&\,\,\frac{d}{dx}(x(1-x^2)\sum_{i=1}^3I_i^{-1}I_i')+\frac{1}{2}x(1-x^2)\sum_{i=1}^3(I_i^{-1}I_i')^2 - 2\sum_{i=1}^3I_i^{-1}I_i'= 0,\\
&-\frac{1}{8}x(1-x^2)^2I_1''+(\frac{1}{2}x^2+\frac{1}{4})(1-x^2)I_1'+\frac{x(1-x^2)^2}{8}I_1^{-1}(I_1')^2+\frac{1}{8}(1-x^4)(I_1^{-1}I_1'+I_2^{-1}I_2'+I_3^{-1}I_3')I_1\\
&-\frac{1}{16}x(1-x^2)^2(I_1^{-1}I_1'+I_2^{-1}I_2'+I_3^{-1}I_3')I_1'-2xI_1+x[4-2(I_3^{-1}I_2+I_2^{-1}I_3)+2(I_2^{-1}I_3^{-1}I_1^2)]=0,\\
&-\frac{1}{8}x(1-x^2)^2I_2''+(\frac{1}{2}x^2+\frac{1}{4})(1-x^2)I_2'+\frac{x(1-x^2)^2}{8}I_2^{-1}(I_2')^2+\frac{1}{8}(1-x^4)(I_1^{-1}I_1'+I_2^{-1}I_2'+I_3^{-1}I_3')I_2\\
&-\frac{1}{16}x(1-x^2)^2(I_1^{-1}I_1'+I_2^{-1}I_2'+I_3^{-1}I_3')I_2'-2xI_2+x[4-2(I_1^{-1}I_3+I_3^{-1}I_1)+2(I_1^{-1}I_3^{-1}I_2^2)]=0,\\
&-\frac{1}{8}x(1-x^2)^2I_3''+(\frac{1}{2}x^2+\frac{1}{4})(1-x^2)I_3'+\frac{x(1-x^2)^2}{8}I_3^{-1}(I_3')^2+\frac{1}{8}(1-x^4)(I_1^{-1}I_1'+I_2^{-1}I_2'+I_3^{-1}I_3')I_3\\
&-\frac{1}{16}x(1-x^2)^2(I_1^{-1}I_1'+I_2^{-1}I_2'+I_3^{-1}I_3')I_3'-2xI_3+x[4-2(I_1^{-1}I_2+I_2^{-1}I_1)+2(I_1^{-1}I_2^{-1}I_3^2)]=0,\\
\end{align*}
for $x\in[0,1]$. Denote $\phi_1=\frac{ I_2}{I_1}$, $\phi_2=\frac{I_3}{I_2}$, $y_1=\log(K)$, $y_2=\log(\phi_1)$ and $y_3=\log(\phi_2)$, so that
\begin{align*}
I_1=(K\phi_1^{-2}\phi_2^{-1})^{\frac{1}{3}},\,\,I_2=(K\phi_1\phi_2^{-1})^{\frac{1}{3}},\,\,I_3=(K\phi_1\phi_2^2)^{\frac{1}{3}}.
\end{align*}

Therefore, for a generalized Berger metric $\hat{g}$ in $(\ref{equn_GBergermetric})$ with $\lambda_1, \lambda_2, \lambda_3$ different from one another, the boundary value problem of the Einstein metrics becomes
\begin{align}
&\label{equn_GBergerEinstein01}y_1''-x^{-1}(1+3x^2)(1-x^2)^{-1}y_1'+\frac{1}{6}(y_1')^2+\frac{1}{3}[(y_2')^2+y_2'y_3'+(y_3')^2]=0,\\
&\label{equn_GBergerEinstein02}y_1''-x^{-1}(5+7x^2)(1-x^2)^{-1}y_1'+\frac{1}{2}(y_1')^2+16(1-x^2)^{-2}[3-2K^{-\frac{1}{3}}(\phi_1^2\phi_2)^{\frac{1}{3}}-2K^{-\frac{1}{3}}(\phi_1^{-1}\phi_2)^{\frac{1}{3}}\\
&-2K^{-\frac{1}{3}}(\phi_1\phi_2^2)^{-\frac{1}{3}}+K^{-\frac{1}{3}}\phi_1^{-\frac{4}{3}}\phi_2^{-\frac{2}{3}}+K^{-\frac{1}{3}}\phi_1^{\frac{2}{3}}\phi_2^{-\frac{2}{3}}+K^{-\frac{1}{3}}\phi_1^{\frac{2}{3}}\phi_2^{\frac{4}{3}}]=0,
\notag \\
&\label{equn_GBergerEinstein03}y_2''-2x^{-1}(1+2x^2)(1-x^2)^{-1}y_2'+\frac{1}{2}y_1'y_2'+32(1-x^2)^{-2}K^{-\frac{1}{3}}[\phi_1^{\frac{2}{3}}\phi_2^{\frac{1}{3}}-\phi_1^{-\frac{1}{3}}\phi_2^{\frac{1}{3}}-\phi_1^{\frac{2}{3}}\phi_2^{-\frac{2}{3}}+\phi_1^{-\frac{4}{3}}\phi_2^{-\frac{2}{3}}]=0,\\
&\label{equn_GBergerEinstein04}y_3''-2x^{-1}(1+2x^2)(1-x^2)^{-1}y_3'+\frac{1}{2}y_1'y_3'+32(1-x^2)^{-2}K^{-\frac{1}{3}}[\phi_1^{-\frac{1}{3}}\phi_2^{\frac{1}{3}}-\phi_1^{-\frac{1}{3}}\phi_2^{-\frac{2}{3}}-\phi_1^{\frac{2}{3}}\phi_2^{\frac{4}{3}}+\phi_1^{\frac{2}{3}}\phi_2^{-\frac{2}{3}}]=0,
\end{align}
for $y_i(x)\in C^{\infty}([0,1])$ for $i=1,2,3$ with the boundary condition
\begin{align}\label{equn_GBergerBV01}
\phi_1(0)=\frac{\lambda_2}{\lambda_1},\,\,\phi_2(0)=\frac{\lambda_3}{\lambda_2},\,\,K(1)=\phi_1(1)=\phi_2(1)=1,\,\,y_i'(0)=y_i'(1)=0,\,\,\,\text{for}\,\,i=1,2,3.
\end{align}
Combining $(\ref{equn_GBergerEinstein01})$ and $(\ref{equn_GBergerEinstein02})$ we have
\begin{align}\label{equn_GBergerEinstein05}
&(y_1')^2-[(y_2')^2+y_2'y_3'+(y_3')^2]-12x^{-1}(1+x^2)(1-x^2)^{-1}y_1'+48(1-x^2)^{-2}[3-2K^{-\frac{1}{3}}(\phi_1^2\phi_2)^{\frac{1}{3}}\\
&-2K^{-\frac{1}{3}}(\phi_1^{-1}\phi_2)^{\frac{1}{3}}-2K^{-\frac{1}{3}}(\phi_1\phi_2^2)^{-\frac{1}{3}}+K^{-\frac{1}{3}}\phi_1^{-\frac{4}{3}}\phi_2^{-\frac{2}{3}}+K^{-\frac{1}{3}}\phi_1^{\frac{2}{3}}\phi_2^{-\frac{2}{3}}+K^{-\frac{1}{3}}\phi_1^{\frac{2}{3}}\phi_2^{\frac{4}{3}}]=0\notag,
\end{align}
Similar as the Berger metric case, if we denote the function on the left hand side of $(\ref{equn_GBergerEinstein05})$ as $\Phi$, take derivative of $\Phi$ and apply $(\ref{equn_GBergerEinstein02})$, $(\ref{equn_GBergerEinstein03})$ and $(\ref{equn_GBergerEinstein04})$, we have that
\begin{align*}
\Phi'+(y_1'-x^{-1}(1-x^2)^{-1}(4+8x^2))\Phi=0.
\end{align*}
Consider $y_1'$ as a given function. By the expansion of the metric and the initial data, similar as above, the equation has a unique solution $\Phi=0$. Therefore, $(\ref{equn_GBergerEinstein02})$, $(\ref{equn_GBergerEinstein03})$ and $(\ref{equn_GBergerEinstein04})$ combining with the expansion of the Einstein metric imply $(\ref{equn_GBergerEinstein05})$. Similarly, any three equations in the system of five equations $(\ref{equn_GBergerEinstein01})-(\ref{equn_GBergerEinstein04})$ and $(\ref{equn_GBergerEinstein05})$ containing at least one of $(\ref{equn_GBergerEinstein03})$ and $(\ref{equn_GBergerEinstein04})$, combining with the initial data imply the other two equations.

\section{Uniqueness of The Solution to The Boundary Value Problem $(\ref{equn_BergerEinstein01})-(\ref{equn_BergerBV01})$}

From now on, we study the uniqueness of solutions to the boundary value problem $(\ref{equn_BergerEinstein01})-(\ref{equn_BergerBV01})$. We start with the monotonicity of $y_1$ and $y_2$ for global solutions on $x\in[0,1]$. Note that for the special case $\phi(0)=1$, uniqueness of the solution is proved in \cite{Andersson-Dahl}\cite{Q}\cite{DJ}\cite{LiQingShi}. For $\phi(0)\neq 1$, by volume comparison theorem,
\begin{align*}
K(0)=\lim_{x\to 0}\frac{\text{det}(\bar{h})}{\text{det}(\bar{h}^{\mathbb{H}^4}(x))}=\lim_{r\to +\infty}\frac{\text{det}(g_r)}{\text{det}(g_r^{\mathbb{H}^4}(r))}<1,
\end{align*}
where \begin{align*}
g^{\mathbb{H}^4}=dr^2+g_r^{\mathbb{H}^4}(r)=x^{-2}(dx^2+\frac{(1-x^2)^2}{4}\bar{h}^{\mathbb{H}^4})
\end{align*}
is the Hyperbolic metric. Moreover it is proved in \cite{LiQingShi} that
\begin{align*}
(\frac{Y(S^3,[\hat{g}])}{Y(S^3,[g^{\mathbb{S}^3}])})^{\frac{3}{2}}\leq K(0)=\lim_{r\to +\infty}\frac{\text{det}(g_r)}{\text{det}(g_r^{\mathbb{H}^4}(r))},
\end{align*}
where $Y(S^3,[\hat{g}])$ is the Yamabe constant of $(S^3,[\hat{g}])$ and $g^{\mathbb{S}^n}$ is the round sphere metric.
\begin{lem}\label{lem_monotonicity01}
For the initial data $\phi(0)\neq 0$, we have $y_1'(x)>0$ for $x\in(0,1)$. Also, $y_2'(x)>0$ and $\phi(0)<\phi(x)<1$ for $x\in(0,1)$ if $\phi(0)<1$; while $y_2'(x)<0$ and $1<\phi(x)<\phi(0)$ for $x\in(0,1)$ if $\phi(0)>1$. That is to say, $K$ and $\phi$ are monotonic on $x\in(0,1)$.
\end{lem}
\begin{proof}
Note that the zeroes of $y_1'$ are discrete on $x\in[0,1]$. Assume that there exists a zero of $y_1'$ on $x\in(0,1)$. Let $x_1$ be the largest zero of $y_1'$ on $x\in(0,1)$. Multiplying $x^{-1}(1-x^2)^2$ on both sides of $(\ref{equn_BergerEinstein01})$ and integrating the equation on $x\in[x_1,1]$, we have
\begin{align}
&\label{equn_singularint21}(x^{-1}(1-x^2)^2y_1')'+x^{-1}(1-x^2)^2[\frac{1}{6}(y_1')^2+\frac{1}{3}(y_2')^2]=0,\\
&\int_{x_1}^1x^{-1}(1-x^2)^2[\frac{1}{6}(y_1')^2+\frac{1}{3}(y_2')^2]dx=0.\notag
\end{align}
Therefore, $y_1'=0$ on $x\in[x_1,1]$. Since $y_1$ is analytic, $y_1'=0$ for $x\in[0,1]$, contradicting with the fact $y_1(0)<y_1(1)$. Therefore, there is no zero of $y_1'$ for $x\in(0,1)$. Therefore, $y_1'>0$ for $x\in (0,1)$.

If there is a zero of $y_2'$ on $x\in(0,1)$, since the zero of $y_2'$ is discrete, assume $x_2$ is the largest one on $x\in(0,1)$. Multiplying $x^{-2}(1-x^2)^3$ on both sides of $(\ref{equn_BergerEinstein03})$ and integrating the equation on $x\in[x_2,1]$, we have
\begin{align}
&\label{equn_y2singularint}(x^{-2}(1-x^2)^3y_2')'+\frac{1}{2}x^{-2}(1-x^2)^3y_1'y_2'+ 32x^{-2}(1-x^2)K^{-\frac{1}{3}}\phi^{-\frac{4}{3}}(1-\phi)=0,\\
&\int_{x_2}^1\frac{1}{2}x^{-2}(1-x^2)^3y_1'y_2'dx= \int_{x_2}^132x^{-2}(1-x^2)K^{-\frac{1}{3}}\phi^{-\frac{4}{3}}(\phi-1)dx.\notag
\end{align}
Since $y_2'=\phi^{-1}\phi'$ has no zeroes on $x\in(x_2,1)$, $\phi$ is monotonic on $(x_2,1)$. Note that we have proved $y_1'>0$ for $x\in(0,1)$. The fact $\phi(1)=1$ implies that $(\phi-1)y_2'<0$ for $x\in(x_2,1)$, contradicting with the integration unless $y_2'=0$ for $x\in(x_2,1)$ which implies that $y_2'=0$ for $x\in[0,1]$ by the differential equation. But it again contradicts with the condition $\phi(0)\neq1$. Therefore, there is no zero of $y_2'$ on $x\in(0,1)$. This completes the proof of the lemma.
\end{proof}
By $(\ref{equn_BergerEinstein04})$ and the initial value condition, we have
\begin{align}
y_1'&\label{equn_y1lowerorder1}=6x^{-1}(1-x^2)^{-1}[1+x^2-\sqrt{(1+x^2)^2+\frac{1}{36}x^2(1-x^2)^2(y_2')^2-\frac{4}{3}x^2(3-4(\phi K)^{-\frac{1}{3}}+K^{-\frac{1}{3}}\phi^{-\frac{4}{3}})}\,\,]\\
&=6x^{-1}(1-x^2)^{-1}[1+x^2-\sqrt{(1-x^2)^2+\frac{1}{36}x^2(1-x^2)^2(y_2')^2+\frac{4}{3}x^2 K^{-\frac{1}{3}}\phi^{-\frac{4}{3}}(4\phi-1)}\,\,].
\end{align}
Since $y_1'>0$ for $x\in(0,1)$, it is clear that
\begin{align}\label{inequn_boundaryYamabeconstant}
3-4(\phi K)^{-\frac{1}{3}}+K^{-\frac{1}{3}}\phi^{-\frac{4}{3}}>\frac{1}{36}x^2(1-x^2)^2(y_2')^2\geq 0\,\, \text{for}\,\, x\in(0,1).
\end{align}
Note that by $(\ref{equn_BergerEinstein02})$ and $(\ref{equn_BergerEinstein03})$,
\begin{align*}
&y_1''(0)=4(3-4K^{-\frac{1}{3}}(0)\phi(0)^{-\frac{1}{3}}+K^{-\frac{1}{3}}(0)\phi(0)^{-\frac{4}{3}}),\\
&y_2''(0)=32K^{-\frac{1}{3}}(0)\phi^{-\frac{4}{3}}(0)(1-\phi(0)).
\end{align*}
Therefore,
\begin{align*}
&\frac{d^2}{dx^2}(3-4(\phi K)^{-\frac{1}{3}}+K^{-\frac{1}{3}}\phi^{-\frac{4}{3}})\big|_{x=0}\\
=&\frac{4}{3}K^{-\frac{1}{3}}\phi^{-\frac{1}{3}}[y_1''(0)+y_2''(0)]-\frac{1}{3}K^{-\frac{1}{3}}(0)\phi^{-\frac{4}{3}}(0)(y_1''(0)+4y_2''(0))\\
=&\frac{4}{3}K^{-\frac{1}{3}}\phi^{-\frac{4}{3}}[(4\phi-1)(3-4(\phi K)^{-\frac{1}{3}}+K^{-\frac{1}{3}}\phi^{-\frac{4}{3}}) - 32 K^{-\frac{1}{3}}\phi^{-\frac{4}{3}}(1-\phi)^2]\big|_{x=0}.
\end{align*}
For $\phi(0)\neq 1$, if $3-4(\phi(0) K(0))^{-\frac{1}{3}}+K^{-\frac{1}{3}}(0)\phi^{-\frac{4}{3}}(0)=0$, then \begin{align*}
\frac{d^2}{dx^2}(3-4(\phi K)^{-\frac{1}{3}}+K^{-\frac{1}{3}}\phi^{-\frac{4}{3}})\big|_{x=0}<0.
\end{align*}
Since $\frac{d}{dx}(3-4(\phi K)^{-\frac{1}{3}}+K^{-\frac{1}{3}}\phi^{-\frac{4}{3}})\big|_{x=0}=0$, we have
\begin{align*}
(3-4(\phi K)^{-\frac{1}{3}}+K^{-\frac{1}{3}}\phi^{-\frac{4}{3}})<0,
\end{align*}
for $x>0$ small, contradicting with $(\ref{inequn_boundaryYamabeconstant})$. Therefore, if $\phi(0)\neq 1$, combining with $(\ref{inequn_boundaryYamabeconstant})$, we have
\begin{align}
3-4(\phi(0) K(0))^{-\frac{1}{3}}+K^{-\frac{1}{3}}(0)\phi^{-\frac{4}{3}}(0)>0.
\end{align}
This gives a lower bound of $y_1(0)$ for $\phi(0)>\frac{1}{4}$. Also, if $\phi(0)>\frac{1}{4}$, by Lemma \ref{lem_monotonicity01} and the equation $(\ref{equn_y1lowerorder1})$ we have for $x\in(0,1)$
\begin{align*}
y_1'<6x^{-1}(1-x^2)^{-1}[1+x^2-\sqrt{(1-x^2)^2}\,\,]=12x(1-x^2)^{-1}.
\end{align*}

\begin{thm}\label{thm_boundarycondition01}
We assume that there are two solutions $(y_{11},y_{12})$ and $(y_{21}, y_{22})$ to the boundary value problem $(\ref{equn_BergerEinstein01})-(\ref{equn_BergerBV01})$, with $y_{11}=\log(K_1),\,y_{12}=\log(\phi_1),\,y_{21}=\log(K_2)$ and $y_{22}=\log(\phi_2)$. Then if $K_1(0)=K_2(0)$, $\phi(0)\neq 1$ and $4>\phi(0)>\frac{1}{4}$, we have $(y_{11},y_{12})=(y_{21},y_{22})$ for $x\in[0,1]$.
\end{thm}
\begin{proof}
In the expansion of $g_x$, denote $g^{(3)}=\frac{1}{4}\bar{h}^{(3)}=\text{diag}(- 2\phi(0)^{-1}a,a,a)$ the trace-free nonlocal term. If $K(0),\,\phi(0)$ and $a$ are fixed, then Biquard \cite{Biquard2} proved that the Einstein metric is unique in a neighborhood of the boundary. If the global solution exists, by analytic extension, it must be unique.

So, now we assume that for $(y_{11}, y_{12})$ and $(y_{21},y_{22})$, the nonlocal terms in the expansion are different, with two different numbers $a_1$ and $a_2$ instead of $a$ in the diagonal of $g^{(3)}$ corresponding to the two metrics. With out loss of generality, assume $a_1<a_2$. Denote $z_1=y_{11}-y_{21}$ and $z_2=y_{12}-y_{22}$.

Assume $y_i$ has the expansion $y_i(x)=y_i(0)+\displaystyle\sum_{k=2}^{+\infty}x^ky_i^{(k)}$ at $x=0$ for $i=1,2$. By taking derivatives on both sides of $(\ref{equn_BergerEinstein01})-(\ref{equn_BergerEinstein03})$, we can solve $y_1^{(k)}(0)$ and $y_2^{(k)}$ inductively using $a$ in $g^{(3)}$. Note that $y_1^{(3)}(0)=0$. Also, the first term that depends on $a$ in the expansion of $y_1$ at $x=0$ is $y_1^{(5)}=-\frac{4}{15}y_2^{(2)}y_2^{(3)}$. By $(\ref{equn_BergerEinstein03})$,
\begin{align}
y_2^{(2)}=16K(0)^{-\frac{1}{3}}\phi(0)^{-\frac{4}{3}}(1-\phi(0)).
\end{align}
Also,
\begin{align}
y_2^{(3)}=I_2^{-1}(0)I_2^{(3)}(0)-I_1^{-1}(0)I_1^{(3)}(0)=12(K(0)\phi(0))^{-\frac{1}{3}}a.
\end{align}
Therefore,
\begin{align*}
y_1^{(5)}=-\frac{256}{5}K(0)^{-\frac{2}{3}}\phi(0)^{-\frac{5}{3}}(1-\phi(0))a.
\end{align*}

Since $(y_{11}(x),y_{12}(x))=(y_{21}(x), y_{22}(x))$ for $x=0,1$, by the mean value theorem there must be zero points of $z_1'$ and $z_2'$ in $x\in(0,1)$. We {\bf claim} that $z_1'$ can not achieve its zero before the first zero of $z_2'$. Assume otherwise. Let $x_1$ be the first zero of $z_1'$ on $(0,1)$. Then $z_1'\neq 0$ and $z_2'\neq 0$ in $(0, x_1)$. Since $z_1'(x)=5(y_{11}^{(5)}-y_{21}^{(5)})x^4+o(x^4)$ and $z_2'=3(y_{21}^{(3)}-y_{22}^{(3)})x^2+o(x^2)$, we have that $z_1'z_2'(1-\phi(0))<0$ for $x\in(0, x_1)$ and $(K_1-K_2)(\phi_1-\phi_2)(1-\phi(0))<0$ for $x\in(0,x_1]$. We substitute the two solutions to $(\ref{equn_y1lowerorder1})$ and take difference of the two equations obtained to have
\begin{align}\label{equn_differencez1}
z_1'=\frac{-[(y_{12}'+y_{22}')z_2'+48(1-x^2)^{-2}\big(K_1^{-\frac{1}{3}}(4\phi_1^{-\frac{1}{3}}-\phi_1^{-\frac{4}{3}})-K_2^{-\frac{1}{3}}(4\phi_2^{-\frac{1}{3}}-\phi_2^{-\frac{4}{3}})\big)]}{\sqrt{36x^{-2}+(y_{12}')^2+48\frac{K_1^{-\frac{1}{3}}(4\phi_1^{-\frac{1}{3}}-\phi_1^{-\frac{4}{3}})}{(1-x^2)^2}}+\sqrt{36x^{-2}+(y_{22}')^2+48\frac{K_2^{-\frac{1}{3}}(4\phi_2^{-\frac{1}{3}}-\phi_2^{-\frac{4}{3}})}{(1-x^2)^2}}}.
\end{align}
By Lemma \ref{lem_monotonicity01}, at $x_1$ the two terms in the numerator on the right hand side of $(\ref{equn_differencez1})$ have the same sign and the right hand side is non-zero, contradicting with $z_1'(x_1)=0$. This proves the {\bf claim}.

Multiplying $x^{-2}(1-x^2)^3K^{\frac{1}{2}}$ on both sides of $(\ref{equn_BergerEinstein03})$, we have
\begin{align}\label{equn_y2singularint2}
(x^{-2}(1-x^2)^3K^{\frac{1}{2}}y_2')'+32x^{-2}(1-x^2)K^{\frac{1}{6}}\phi^{-\frac{4}{3}}(1-\phi)=0.
\end{align}
Substituting the two solutions to $(\ref{equn_y2singularint2})$ and taking difference of the equations obtained, we have
\begin{align*}
[x^{-2}(1-x^2)^3(K_1^{\frac{1}{2}}y_{12}'-K_2^{\frac{1}{2}}y_{22}')]'+32x^{-2}(1-x^2)[K_1^{\frac{1}{6}}\phi_1^{-\frac{4}{3}}(1-\phi_1)-K_2^{\frac{1}{6}}\phi_2^{-\frac{4}{3}}(1-\phi_2)]=0.
\end{align*}
Let $x_2$ be the first zero of $z_2'$ on $x\in(0,1)$. Then by the {\bf claim} above, $z_1'(x)\neq0$ and also the expression of $y_1^{(5)}$ tells that $z_1'(x)$ is of the same sign as $-(1-\phi(0))(a_1-a_2)$, for $x\in(0,x_2)$. 
Integrating the equation on $x\in[0,x_2]$, we have
\begin{align*}
&x_2^{-1}(1-x^2)^3(K_1^{\frac{1}{2}}(x_2)-K_2^{\frac{1}{2}}(x_2))y_{12}'(x_2)-12K(0)^{\frac{1}{6}}\phi(0)^{-\frac{1}{3}}(a_1-a_2)\\
&+32\int_0^{x_2}x^{-2}(1-x^2)^2[K_1^{\frac{1}{6}}(\phi_1^{-\frac{4}{3}}-\phi_1^{-\frac{1}{3}})-K_2^{\frac{1}{6}}(\phi_2^{-\frac{4}{3}}-\phi_2^{-\frac{1}{3}})]dx=0,
\end{align*}
Note that the three terms on the left hand side are of the same sign and non-zero, which is a contradiction. Therefore, $z_1'(x)\neq 0$ and $z_2'(x)\neq0$ for $x\in(0,1)$, which is a contradiction to the mean value theorem for $y_1$ and $y_2$ on $x\in[0,1]$. Therefore, $a_1=a_2$ and $(y_{11},y_{12})=(y_{21},y_{22})$ on $x\in[0,1]$. This completes the proof of the theorem.
\end{proof}
Now to show the uniqueness of the solution to $(\ref{equn_BergerEinstein01})-(\ref{equn_BergerBV01})$, we only need to rule out the possibility of the existence of two solutions $(y_{11}, y_{12})$ and $(y_{21}, y_{22})$ with $y_{11}=\log(K_1),\,y_{12}=\log(\phi_1),\,y_{21}=\log(K_2)\,$ and $y_{22}=\log(\phi_2)$ such that $K_1(0)\neq K_2(0)$, for $\phi(0)\neq 1$ and $4>\phi(0)>\frac{1}{4}$. Without loss of generality, assume $K_1(0)> K_2(0)$.

Denote $z_1=y_{11}-y_{21}$ and $z_2=y_{12}-y_{22}$. As a preparation, we start with a lemma about zeroes of $z_1'$ and $z_2'$.

\begin{lem}\label{lem_zeroz1z2}
For any two zeroes $0<x_1<x_2\leq 1$ of $z_1'$ so that there is no zero of $z_1'$ on the interval $x\in(x_1,x_2)$, there exists a point $x_3\in(x_1,x_2)$ so that
\begin{align}\label{inequn_signz201}
(y_{12}'+y_{22}')z_1'z_2'\big|_{x=x_3}<0.
\end{align}
Also, for any zero $0<x_2\leq 1$ of $z_1'$, there exists $\varepsilon>0$ so that for any $x_2-\varepsilon < x <x_2$, we have
\begin{align}\label{inequn_signz202}
(y_{12}'(x)+y_{22}'(x))z_1'(x)z_2'(x)>0.
\end{align}
\end{lem}
\begin{proof}
We substitute the solutions $(y_{11},y_{12})$ and $(y_{21},y_{22})$ to  $(\ref{equn_singularint21})$, take difference of the two equations obtained and integrate it on the interval $x\in[x_1,x_2]$ to have
\begin{align*}
\int_{x_1}^{x_2}x^{-1}(1-x^2)^2[\frac{1}{6}(y_{11}'+y_{21}')z_1'+\frac{1}{3}(y_{21}'+y_{22}')z_2'] dx=0.
\end{align*}
Since $y_{11}',y_{21}'>0$, we have that $(\ref{inequn_signz201})$ holds for some point $x_3\in(x_1,x_2)$.

Multiplying $x(1-x^2)$ to $(\ref{equn_BergerEinstein01})$, we have
\begin{align}\label{equn_regularint01}
-2y_1'+(x(1-x^2)y_1')'+x(1-x^2)[\frac{1}{6}(y_1')^2+\frac{1}{3}(y_2')^2]=0.
\end{align}
Now we substitute the solutions $(y_{11},y_{12})$ and $(y_{21},y_{22})$ to  $(\ref{equn_regularint01})$, take difference of the two equations and integrate it on the interval $x\in[x_2-\varepsilon,x_2]$ to have
\begin{align*}
2(z_1(x_2)-z_1(x_2-\varepsilon))+(x(1-x^2)z_1')\big|_{x=x_2-\varepsilon}=\int_{x_2-\varepsilon}^{x_2}x(1-x^2)[\frac{1}{6}(y_{11}'+y_{21}')z_1'+\frac{1}{3}(y_{21}'+y_{22}')z_2'] dx.
\end{align*}
It is clear that the two terms on the left hand side are of the same sign for $\varepsilon>0$ small, since zeroes of $z_1'$ are discrete. Note that the first integral term on the right hand side is a higher order small term comparing to the second term on the left hand side, as the positive constant $\varepsilon\to 0$. Therefore, since zeroes of $z_2'$ are discrete, for $\varepsilon>0$ small  $(\ref{inequn_signz202})$ holds.
\end{proof}
\begin{thm}\label{thm_BergermetricStability}
There exists at most a unique solution to the boundary value problem $(\ref{equn_BergerEinstein01})-(\ref{equn_BergerBV01})\,$ for $\frac{1}{4}<\phi(0)<1$ and $1<\phi(0)<4$.
\end{thm}
\begin{proof}
We argue by contradiction. Assume that there are two solutions $(y_{11}, y_{12})$ and $(y_{21}, y_{22})$ with $y_{11}=\log(K_1),\,y_{12}=\log(\phi_1),\,y_{21}=\log(K_2)\,$ and $y_{22}=\log(\phi_2)$. By Theorem \ref{thm_boundarycondition01}, $K_1(0)\neq K_2(0)$. Without loss of generality, assume $K_1(0)> K_2(0)$. Denote $z_1=y_{11}-y_{21}$ and $z_2=y_{12}-y_{22}$.

Let $x_1$ be the largest zero point of $z_1'$ on $x\in[0,1)$. We {\bf claim} that there exists a zero of $z_2'$ on $x\in( x_1, 1)$. First, if $x_1=0$, since $z_2(0)=z_2(1)=0$, by the mean value theorem, there exists $x_2\in(0,1)$ such that $z_2'(x_2)=0$. Second, if $x_1>0$, then by Lemma \ref{lem_zeroz1z2}, $z_2'$ changes signs on $x\in[x_1,1]$ and there exists $x_2\in(x_1,1)$ such that $z_2'(x_2)=0$. This proves the {\bf claim}. Let $x_2$ be the largest zero of $z_2'$ on $x\in(x_1, 1)$.

Multiplying $x^{-2}(1-x^2)^3$ on both sides of $(\ref{equn_BergerEinstein03})$ we have
\begin{align}\label{equn_preinterg02}
(x^{-2}(1-x^2)^3y_2')'+\frac{1}{2}x^{-2}(1-x^2)^3y_1'y_2'+32x^{-2}(1-x^2)K^{-\frac{1}{3}}\phi^{-\frac{4}{3}}(1-\phi)=0.
\end{align}
Now we substitute the two solutions into $(\ref{equn_preinterg02})$, take difference of the two equations obtained and integrate it on the interval $x\in[x_2,1]$ to have
\begin{align}
&\label{equn_singularint02}32\int_{x_2}^1x^{-2}(1-x^2)[(K_1^{-\frac{1}{3}}-K_2^{-\frac{1}{3}})\phi_{1}^{-\frac{4}{3}}(1-\phi_1)+K_2^{-\frac{1}{3}}((\phi_1^{-\frac{4}{3}}-\phi_1^{-\frac{1}{3}})-(\phi_2^{-\frac{4}{3}}-\phi_2^{-\frac{1}{3}}))] dx\\
&+\int_{x_2}^1\frac{1}{2}x^{-2}(1-x^2)^3(z_1'y_{12}'+y_{21}'z_2')dx=0.\notag
\end{align}
Since $z_1(1)=z_2(1)=0$, we have $(K_1-K_2)z_1'<0$ and $(\phi_1-\phi_2)z_2'<0$ on $x\in(x_2,1)$. Note that $y_{i1}'>0$ and $(1-\phi_i)y_{i2}'>0$ for $i=1,\,2$ in $(0,1)$ by Lemma \ref{lem_monotonicity01}. Since there is no zero of $z_1'$ and $z_2'$ on $x\in (x_2,1)$, by Lemma \ref{lem_zeroz1z2} we have that $z_1'y_{12}'$, $y_{21}'z_2'$, $(K_1^{-\frac{1}{3}}-K_2^{-\frac{1}{3}})(1-\phi_1)$ and $((\phi_1^{-\frac{4}{3}}-\phi_1^{-\frac{1}{3}})-(\phi_2^{-\frac{4}{3}}-\phi_2^{-\frac{1}{3}}))$ are of the same sign on $x\in(x_2,1)$, contradicting with $(\ref{equn_singularint02})$. Note that a solution to the boundary value problem is given explicitly by Pedersen in \cite{Pedersen}, therefore it must be the unique solution when it has non-positive sectional curvature. This completes the proof of the theorem.
\end{proof}

\begin{proof}{\it Proof of Theorem \ref{thm_someBergermetric}}
For the case $\phi(0)=1$ i.e., $\lambda_1=\lambda_2$ so that the conformal infinity is the round sphere, the theorem has been proved in \cite{Andersson-Dahl}\cite{Q}\cite{DJ}\cite{LiQingShi}.

Now we assume that $\lambda_1\neq \lambda_2.\,$ Let $q\in \partial M=S^3$ be a fixed point. By Theorem \ref{thm_definingfunction} and the discussion in Section \ref{section_CCEinsteinHCITY}, we have that the Einstein equations with prescribed conformal infinity which is the conformal class of the Berger metric, is equivalent to the boundary value problem $(\ref{equn_BergerEinstein01})-(\ref{equn_BergerBV01})$ along the geodesic connecting the center of gravity $p_0$ of $(M, g)$ and $q$. Then by Theorem \ref{thm_BergermetricStability}, up to isometries, the conformally compact Einstein metric is unique.
\end{proof}

 \end{document}